\documentclass[a4paper,12pt]{article}
\usepackage{latexsym, amssymb, amsfonts,amsmath}
\usepackage{amssymb,latexsym}
\usepackage{amsfonts}
\usepackage{amsmath,amsthm,graphicx}
\newcommand{\be}{\begin{equation}}
\newcommand{\ee}{\end{equation}}
\newcommand{\bea}{\begin{eqnarray}}
\newcommand{\eea}{\end{eqnarray}}
\newcommand{\bean}{\begin{eqnarray*}}
\newcommand{\eean}{\end{eqnarray*}}
\newcommand{\brray}{\begin{array}}
\newcommand{\erray}{\end{array}}
\newcommand{\ben}{\begin{equation}{nonumber}}
\newcommand{\een}{\end{equation}{nonumber}}

\newtheorem{dfn}{Definition}[section]
\newtheorem{thm}[dfn]{Theorem}
\newtheorem{lmma}[dfn]{Lemma}

\newtheorem{ppsn}[dfn]{Proposition}
\newtheorem{crlre}[dfn]{Corollary}
\newtheorem{xmpl}[dfn]{Example}
\newtheorem{rmrk}[dfn]{Remark}
\newcommand{\bdfn}{\begin{dfn}}
\newcommand{\bthm}{\begin{thm}}

\numberwithin{equation}{section}

\newcommand{\blr}{\begin{list}{$($\roman{cnt1}$)$} {\usecounter{cnt1}
        \setlength{\topsep}{0pt} \setlength{\itemsep}{0pt}}}
\newcommand{\bla}{\begin{list}{$($\alph{cnt2}$)$} {\usecounter{cnt2}
       \setlength{\topsep}{0pt} \setlength{\itemsep}{0pt}}}
\newcommand{\bln}{\begin{list}{$($\arabic{cnt3}$)$} {\usecounter{cnt3}
                \setlength{\topsep}{0pt} \setlength{\itemsep}{0pt}}}
\newcommand{\el}{\end{list}}

\newcommand{\blmma}{\begin{lmma}}
\newcommand{\bppsn}{\begin{ppsn}}
\newcommand{\bcrlre}{\begin{crlre}}
\newcommand{\bxmpl}{\begin{xmpl}}
\newcommand{\brmrk}{\begin{rmrk}}
\newcommand{\edfn}{\end{dfn}}
\newcommand{\ethm}{\end{thm}}
\newcommand{\elmma}{\end{lmma}}

\newcommand{\eppsn}{\end{ppsn}}
\newcommand{\ecrlre}{\end{crlre}}
\newcommand{\exmpl}{\end{xmpl}}
\newcommand{\ermrk}{\end{rmrk}}

\def\a*{{\cal A}_{h,*}}
\def\B{{\cal B}(h)}
\def\B1{{\cal B}_1(h)}
\def\b{{\cal B}^{\rm s.a.}(h)}
\def\b1{{\cal B}^{\rm s.a.}_1(h)}

\newcommand{\suml}{\sum \limits}
\newcommand{\itt}{\int \limits}

\begin{document}

\begin{center}
\Large{\bf{Third Order Trace Formula}}\\
\vspace{0.15in}
{\large Arup Chattopadhyay {\footnote {J.N.Centre for Advanced Scientific Research, Bangalore-560064, INDIA; ~~arup@jncasr.ac.in}} }
{\large and ~~~Kalyan B. Sinha {\footnote { J.N.Centre for Advanced Scientific Research and Indian Institute of Science, Bangalore-560064, INDIA; kbs@jncasr.ac.in}}}\\
\end{center}
\vspace{0.15in}
\begin{abstract}
In \cite{DykemmaSkripkka}, Dykema and Skripka showed the existence of higher order spectral shift functions when the unperturbed self-adjoint operator is
bounded and the perturbations is Hilbert-Schmidt. In this article, we give a different  proof for the existence of spectral shift function for the third
order when the unperturbed operator is self-adjoint (bounded or unbounded, but bounded below). 
\end{abstract}


operators.

\section{Introduction.}\label{sec: intro}
\emph{ \textbf{Notations:}} Here, $\mathcal{H}$ will denote the separable Hilbert space we work in; $\mathcal{B}(\mathcal{H})$, $\mathcal{B}_p(\mathcal{H})$ $[p\geq 1]$,
the set of bounded, Schatten p- class operators in $\mathcal{H}$ respectively with  $\|.\|, \|.\|_p$ as 
the associated norms. In particular $\mathcal{B}_{1}(\mathcal{H})$ and $\mathcal{B}_{2}(\mathcal{H})$ are known as the set of trace class and Hilbert-Schmidt class operators in
$\mathcal{H}$. Let $A$ be a  self-adjoint operator in $\mathcal{H}$ with $\sigma(A)$ as the spectra and $E_{A}(\lambda)$ the spectral family. 
The symbols $\textup{Dom}(A),$ $\textup{Ker}(A),$ $\textup{Ran}(A)$ and $\textup{Tr} A$ denote the domain, kernel, range and trace of the operator $A$ respectively.

Let $A$ (possibly unbounded) and $V$ be two self-adjoint operators in $\mathcal{H}$ such that $V \in \mathcal{B}_1(\mathcal{H}),$
then Krein \cite{mgkrein} proved that there exists a unique real-valued $L^1(\mathbb{R})$- function $\xi$ 
with support in the interval $[a,b]$ ( where $a=min\{\inf \sigma(A+V), \inf \sigma(A)\}$ and $b=max\{\sup \sigma(A+V), $ $\sup \sigma(A)\}$ ) such that 
\begin{equation}\label{eq: kreinintro}
 \textup{Tr}\left[\phi\left(A+V\right)-\phi\left(A\right)\right] = \int_{a}^{b} \phi^{\prime}(\lambda)\xi(\lambda)d\lambda,
\end{equation}
for a large class of functions $\phi$ . The function $\xi$ is known as Krein's spectral shift function and the relation \eqref{eq: kreinintro} is called
Krein's trace formula. In 1985, Voiculescu approached the trace formula \eqref{eq: kreinintro} from a different
direction. Later Voiculescu \cite{Voiculescu}, and Sinha and Mohapatra (\cite{Sinhamohapatra}, \cite{smunitary}) proved that
\begin{equation}
 \textup{Tr}\left[\phi \left(A+V\right)-\phi \left(A\right)\right] = \lim_{n\longrightarrow \infty} \textup{Tr}_n\left[\phi \left((A+V)_n\right)-\phi \left(A_{n}\right)\right] = \int \phi'(\lambda) \xi(\lambda) d\lambda,
\end{equation}
by adapting the Weyl-von Neumann's theorem (where $\phi(.)$ is a suitable function and 

\hspace{-0.8cm} $(A+V)_n,A_{n}$ are finite dimensional approximations of $(A+V)$ and $A$ respectively and
$\textup{Tr}_n$ is the associated finite dimensional trace).In \cite{kopleinko}, Koplienko considers instead $ \phi(A+V) - \phi(A) - D^{(1)}\phi(A) (V),$
where $D^{(1)}\phi(A)$ denotes the first order Frechet derivative of $\phi$ at $A$ \cite{Bhatia} and finds a trace formula for this expression. 
If $V\in \mathcal{B}_2(\mathcal{H}),$ then Koplienko's formula asserts that there exists a unique function $\eta \in L^1(\mathbb{R})$ such that
\begin{equation}\label{eq: intkopeq}
 \textup{Tr}\{\phi(A+V) - \phi(A) - D^{(1)}\phi(A) (V)\} = \int_{-\infty}^{\infty} \phi''(\lambda)\eta(\lambda)d\lambda
\end{equation}
for rational functions $\phi$ with poles off $\mathbb{R}$. In \cite{chttosinha}, Koplienko trace formula was derived using finite dimensional approximation method,
while Dykema and Skripka \cite{DykemmaSkripkka} obtained the formula \eqref{eq: intkopeq} in the semi-finite von Neumann algebra setting and also studied
the existence of higher order spectral shift function. In $\left( \cite{DykemmaSkripkka}, \textup{Theorem 5.1} \right)$, Dykema and Skripka showed that 
for a self-adjoint operator $A$ (possibly unbounded) and a self-adjoint operator $V\in \mathcal{B}_2(\mathcal{H})$, the following assertions hold:
\vspace{0.2in}

\hspace{-0.8cm} \emph{ (i) There is a unique finite real-valued measure} $\nu_3$ \emph{on} $\mathbb{R}$ \emph{ such that the trace formula} 
\begin{equation}\label{eq: intfsteq}
\textup{Tr}\{\phi(A+V) - \phi(A) - D^{(1)}\phi(A)  (V) - \frac{1}{2}D^{(2)}\phi(A) (V,V)\} = \int_{-\infty}^{\infty} \phi'''(\lambda)d\nu_3(\lambda),
\end{equation}
\emph{ holds for suitable functions $\phi,$ where  $D^{(2)}\phi(A)$ is the second order Frechet derivative of $\phi$ at } $A$ \cite{Bhatia}.
\emph{The total variation of $\nu_3$ is bounded by} $\frac{1}{3!}\|V\|_2^3$.
\vspace{0.2in}

\hspace{-0.8cm} \emph{ (ii)  If, in addition, $A$ is bounded, then $\nu_3$ is absolutely continuous.}
\vspace{0.2in}
 
It is noted that there has been a more recent preprint by Potapov , Skripka  and Sukochev \cite{potapovsukochevskripka} in which similar and further
results have been announced. 

This paper is organized as follows. In section 2 , we establish the formula \eqref{eq: intfsteq} for bounded self-adjoint case and 
section 3 is devoted to the unbounded self-adjoint case.

\section{Bounded Case}

The next three lemmas are preparatory for the proof of the main theorem of this section, theorem \ref{thm: mainbddthm}.
\begin{lmma}\label{lmma: combinatory}
Let, for a given $n\in \mathbb{N}$, $\{a_k\}_{k=0}^{n-1}$ be a sequence of complex numbers such that $a_{n-k-1} = a_k$. Then
$$ \sum_{j=0}^{n-1}~\sum_{k=0}^{n-j-1}a_k + \sum_{j=1}^n~\sum_{k=0}^{j-1}a_k = (n+1) \sum_{k=0}^{n-1}a_k.$$
\end{lmma}
\begin{proof}
By changing the indices of summation and using the fact  $a_{n-k-1} = a_k$, we get that
\begin{equation*}
\begin{split}
\sum_{j=0}^{n-1}~\sum_{k=0}^{n-j-1}a_k + \sum_{j=1}^n~\sum_{k=0}^{j-1}a_k = \sum_{j=0}^{n-1}~\sum_{k=j}^{n-1}a_{n-k-1} + \sum_{j=0}^{n-1}~\sum_{k=0}^{j}a_k = \sum_{j=0}^{n-1}~\sum_{k=j}^{n-1}a_{k} + \sum_{j=0}^{n-1}~\sum_{k=0}^{j}a_k \\
& \hspace{-10.8cm} = \sum_{j=0}^{n-1}a_{j} + \sum_{j=0}^{n-1}~\sum_{k=0}^{n-1}a_k = \sum_{j=0}^{n-1}a_{j} + n\sum_{k=0}^{n-1}a_{k} = (n+1) \sum_{k=0}^{n-1}a_{k}.
\end{split}
\end{equation*}
~~~~~~~~~~~~~~~~~~~~~~~~~~~~~~~~~~~~~~~~~~~~~~~~~~~~~~~~~~~~~~~~~~~~~~~~~~~~~~~~~~~~~~~~~~~~~~~~~~~~~~~~~~~~~~~~~~~~~~~~~~~~~~~~~~~~~~~~~~~~~\end{proof}
\begin{lmma}\label{lmma: bddexplmma}
Let $A$ and $V$ be two bounded self-adjoint operators in an infinite dimensional Hilbert space $\mathcal{H}$ such that $V\in \mathcal{B}_3(\mathcal{H})$. Let
$p(\lambda) = \lambda^r ~( r\geq 0)$.Then
\begin{equation}\label{eq: bddlmmamaineq}
\begin{split}
\textup{Tr}\left[(A+V)^r - A^r -D^{(1)}(A^r)(V) - \frac{1}{2}D^{(2)}(A^r)(V,V)\right] \\
& \hspace{-6cm} = r\sum_{k=0}^{r-2} \int_0^1 ds\int_0^s d\tau ~~\textup{Tr}\left[VA_{\tau}^{r-k-2}VA_{\tau}^k-VA^{r-k-2}VA^k\right],
\end{split}
\end{equation}
where $A_{\tau} = A + \tau V$ and $0\leq \tau \leq 1.$
\end{lmma}
\begin{proof}
 For $X\in \mathcal{B}(\mathcal{H})$, ~~$p(A+X) - p(A) = \suml_{j=0}^{r-1} (A+X)^{r-j-1} X A^j$ and hence
$$\left\|p(A+X)-p(A)-\sum_{j=0}^{r-1} A^{r-j-1}XA^j\right\|\leq \sum_{j=0}^{r-2}~\sum_{k=0}^{r-j-2} \|A+X\|^{r-j-k-2}\|X\|\|A\|^k\|X\|\|A\|^j,$$ 
proving that $D^{(1)}(A^r)(X) = \suml_{j=0}^{r-1} A^{r-j-1}XA^j$.
\vspace{0.1in}

Again for $X,Y\in \mathcal{B}(\mathcal{H})$,
\begin{equation*}
\begin{split}
D^{(1)}((A+X)^r)(Y) -D^{(1)}(A^r)(Y)\\
& \hspace{-5cm} = \sum_{j=0}^{r-1} (A+X)^{r-j-1}Y(A+X)^j - \sum_{j=0}^{r-1} A^{r-j-1}YA^j\\
& \hspace{-5cm} = \sum_{j=0}^{r-1} \left[(A+X)^{r-j-1}-A^{r-j-1}\right]Y(A+X)^j + \sum_{j=0}^{r-1} A^{r-j-1}Y\left[(A+X)^j-A^j\right]\\
& \hspace{-5cm} = \sum_{j=0}^{r-2}~\sum_{k=0}^{r-j-2} (A+X)^{r-j-k-2}XA^kY(A+X)^j + \sum_{j=1}^{r-1}~\sum_{k=0}^{j-1} A^{r-j-1}Y(A+X)^kXA^{j-k-1},
\end{split}
\end{equation*}
leading to the estimate
\begin{equation*}
\begin{split}
\|D^{(1)}((A+X)^r)(Y) -D^{(1)}(A^r)(Y) \\
& \hspace{-5cm} - \left( \sum_{j=0}^{r-2}~\sum_{k=0}^{r-j-2} A^{r-j-k-2}XA^kYA^j + \sum_{j=1}^{r-1}~\sum_{k=0}^{j-1} A^{r-j-1}YA^kXA^{j-k-1}\right)\|= \bigcirc (\|X\|^2)
\end{split}
\end{equation*}
for $\|X\|\leq 1$, proving that
\begin{equation}\label{eq: secondfrechet} 
D^{(2)}(A^r)(X,Y) = \sum_{j=0}^{r-2}~\sum_{k=0}^{r-j-2} A^{r-j-k-2}XA^kYA^j + \sum_{j=1}^{r-1}~\sum_{k=0}^{j-1} A^{r-j-1}YA^kXA^{j-k-1}.
\end{equation}
Recall that $A_s=A+sV\in \mathcal{B}_{s.a.}(\mathcal{H}) ~~(0\leq s\leq 1),$ and a similar calculation shows that the map
$[0,1]\ni s\longmapsto A_s^r$ is continuously differentiable in norm-topology and 
$$\frac{d}{ds}(A_s^r) = \sum_{j=0}^{r-1} A_s^{r-j-1}VA_s^j = \sum_{j=0}^{r-1} A_s^jVA_s^{r-j-1}.$$
Hence
\begin{equation*}
\begin{split}
(A+V)^r - A^r - D^{(1)}(A^r)(V) = \int_0^1 ds \frac{d}{ds}(A_s^r) - D^{(1)}(A^r)(V)\\
& \hspace{-10cm} = \int_0^1ds \sum_{j=0}^{r-1}\left(A_s^{r-j-1}VA_s^j - A^{r-j-1}VA^j\right) = \int_0^1ds \sum_{j=0}^{r-1}~\int_0^s d\tau \frac{d}{d\tau}\left(A_{\tau}^{r-j-1}VA_{\tau}^j\right),
\end{split}
\end{equation*}
which by an application of Leibnitz's rule reduces to
$$\int_0^1ds\int_0^sd\tau\left(\sum_{j=0}^{r-2}~\sum_{k=0}^{r-j-2}A_{\tau}^{r-j-k-2}VA_{\tau}^kVA_{\tau}^j + \sum_{j=1}^{r-1}~\sum_{k=0}^{j-1}A_{\tau}^{r-j-1}VA_{\tau}^kVA_{\tau}^{j-k-1}\right)$$
and using \eqref{eq: secondfrechet}, we get
\begin{equation}\label{eq: frecetdiff}
\begin{split}
  (A+V)^r - A^r -D^{(1)}(A^r)(V) - \frac{1}{2}D^{(2)}(A^r)(V,V)\\
& \hspace{-8cm} = \int_0^1ds\int_0^sd\tau\{~\sum_{j=0}^{r-2}~\sum_{k=0}^{r-j-2}A_{\tau}^{r-j-k-2}VA_{\tau}^kVA_{\tau}^j + \sum_{j=1}^{r-1}~\sum_{k=0}^{j-1}A_{\tau}^{r-j-1}VA_{\tau}^kVA_{\tau}^{j-k-1}\\
& \hspace{-5cm} - \sum_{j=0}^{r-2}~\sum_{k=0}^{r-j-2}A^{r-j-k-2}VA^kVA^j - \sum_{j=1}^{r-1}~\sum_{k=0}^{j-1}A^{r-j-1}VA^kVA^{j-k-1}~\}.
\end{split}
\end{equation}
Let us denote the sum of the first and third term inside the integral in \eqref{eq: frecetdiff} to be
\begin{equation*}
\begin{split}
\textup{I}_1 \equiv \sum_{j=0}^{r-2}~\sum_{k=0}^{r-j-2}\left[A_{\tau}^{r-j-k-2}VA_{\tau}^kVA_{\tau}^j - A^{r-j-k-2}VA^kVA^j\right]\\ 
& \hspace{-9.5cm} = \sum_{j=0}^{r-2}~\sum_{k=0}^{r-j-2}\left[A_{\tau}^{r-j-k-2} - A^{r-j-k-2}\right]VA_{\tau}^kVA_{\tau}^j + \sum_{j=0}^{r-2}~\sum_{k=0}^{r-j-2} A^{r-j-k-2}V\left[A_{\tau}^k - A^k\right]VA_{\tau}^j\\
& \hspace{-8cm} + \sum_{j=0}^{r-2}~\sum_{k=0}^{r-j-2} A^{r-j-k-2}VA^kV\left[A_{\tau}^j - A^j\right] \in \mathcal{B}_1(\mathcal{H}),
\end{split}
\end{equation*}
since $V\in \mathcal{B}_3(\mathcal{H})$ and $A_{\tau}^k-A^k\in \mathcal{B}_3(\mathcal{H}) ~~\forall \tau \in [0,1]$ and $k\in \{0,1,2,3,......\}$. Thus by
the cyclicity of trace , we have that 
$$\textup{Tr}(\textup{I}_1) = \sum_{j=0}^{r-2}~\sum_{k=0}^{r-j-2}\textup{Tr}\left[A_{\tau}^{r-k-2}VA_{\tau}^kV - A^{r-k-2}VA^kV\right].$$
Again if we set the sum of the second and fourth term inside the integral in \eqref{eq: frecetdiff} to be 
$$\textup{I}_2\equiv \sum_{j=1}^{r-1}~\sum_{k=0}^{j-1}\left[A_{\tau}^{r-j-1}VA_{\tau}^kVA_{\tau}^{j-k-1} - A^{r-j-1}VA^kVA^{j-k-1}\right]\in \mathcal{B}_1(\mathcal{H}),$$
 a similar calculation shows that
$$\textup{Tr}(\textup{I}_2) = \sum_{j=1}^{r-1}~\sum_{k=0}^{j-1}~\textup{Tr}\left[A_{\tau}^{r-k-2}VA_{\tau}^kV - A^{r-k-2}VA^kV\right].$$
By applying Lemma \ref{lmma: combinatory} with $n=r-1$ and $a_k = \textup{Tr}\left[A_{\tau}^{r-k-2}VA_{\tau}^kV - A^{r-k-2}VA^kV\right]$ and using the 
cyclicity of trace, we conclude that
\begin{equation}\label{eq: i1i2eu}
\begin{split}
\textup{Tr}(\textup{I}_1) + \textup{Tr}(\textup{I}_2) = r \sum_{k=0}^{r-2}~\textup{Tr}\left[A_{\tau}^{r-k-2}VA_{\tau}^kV - A^{r-k-2}VA^kV\right]\\
& \hspace{-7.6cm} = r \sum_{k=0}^{r-2}~\textup{Tr}\left[VA_{\tau}^{r-k-2}VA_{\tau}^k - VA^{r-k-2}VA^k\right].
\end{split}
\end{equation}
Hence combining \eqref{eq: frecetdiff} and \eqref{eq: i1i2eu}, we get the required expression \eqref{eq: bddlmmamaineq}.
~~~~~~~\end{proof}

\begin{lmma}\label{lmma: decomp}
Let $B$ be a bounded operator in an infinite dimensional Hilbert space $\mathcal{H}$(i.e. $B\in \mathcal{B}(\mathcal{H})$). Define 
$\mathcal{M}_{B}:\mathcal{B}_2(\mathcal{H})\longmapsto \mathcal{B}_2(\mathcal{H})$ ( looking upon $\mathcal{B}_2(\mathcal{H}) \equiv \widetilde{\mathcal{H}}$
as a Hilbert space with inner product given by trace i.e.$\langle X,Y\rangle _2 = \textup{Tr}\{X^*Y\}$ for $X,Y\in\mathcal{B}_2(\mathcal{H})$) by $\mathcal{M}_{B}(X) = BX-XB ~;~ X\in \mathcal{B}_2(\mathcal{H})$. Then
\vspace{0.1in}

(i) $\mathcal{M}_{B}$ is a bounded operator on $\widetilde{\mathcal{H}}$ (i.e.~~$\mathcal{M}_{B}\in \mathcal{B}(\widetilde{\mathcal{H}}))$ with $\mathcal{M}^*_{B} = \mathcal{M}_{B^*}$.
\vspace{0.1in}

(ii) $\textup{Ker}\left(\mathcal{M}_{B}\right)$ and its orthogonal complement $ \overline{\textup{Ran}\left(\mathcal{M}_{B^*}\right)}$ in  $\widetilde{\mathcal{H}}$ are left invariant 
by left and right multiplication by $B^n$ and $(B^*)^n ~(n=1,2,3,...)$ respectively.
\vspace{0.1in}

(iii) $\widetilde{\mathcal{H}} = \textup{Ker}\left(\mathcal{M}_{B}\right) \bigoplus \overline{\textup{Ran}\left(\mathcal{M}_{B^*}\right)}$ ; $\mathcal{B}_2(\mathcal{H})\ni X = X_1 \oplus X_2 $, where $X_1\in \textup{Ker}\left(\mathcal{M}_{B}\right)$
and $X_2\in \overline{\textup{Ran}\left(\mathcal{M}_{B^*}\right)}$.
\vspace{0.1in}

(iv) If $\textup{Ker}\left(\mathcal{M}_{B}\right) = \textup{Ker}\left(\mathcal{M}_{B^*}\right)$, then $\textup{Ker}\left(\mathcal{M}_{B}\right)$ and $ \overline{\textup{Ran}\left(\mathcal{M}_{B}\right)}$
are generated by their self-adjoint elements and for $X\in \widetilde{\mathcal{H}}$, we have $\left(X^*\right)_1 = X_1^*$ and $\left(X^*\right)_2 = X_2^*$, where
$X = X_1 \oplus X_2$ and $X^* = \left(X^*\right)_1 \oplus \left(X^*\right)_2$ are the respective decompositions of $X$ and $X^*$ in $\widetilde{\mathcal{H}}$.
\vspace{0.1in}

(v) If $\textup{Ker}\left(\mathcal{M}_{B}\right) = \textup{Ker}\left(\mathcal{M}_{B^*}\right)$, then for $X = X^*\in\widetilde{\mathcal{H}}$, ~$X = X_1 \oplus X_2$
with $X_1$ and $X_2$  both self-adjoint.
\vspace{0.1in}

(vi) (a) For $B = B^*\in\mathcal{B}(\mathcal{H}),$ $\mathcal{M}_{B}$ is self-adjoint in $\widetilde{\mathcal{H}}$ and for $X = X^*\in\widetilde{\mathcal{H}}$, we have $X_1 = X_1^*$ and $X_2=X_2^*$.
\vspace{0.1in}

(b) For $B = (A+\textup{i})^{-1}$ ( where $A$ is an unbounded self-adjoint operator in $\mathcal{H}$), $\mathcal{M}_{B}$ is bounded normal in $\widetilde{\mathcal{H}}$ and for $X = X^*\in\widetilde{\mathcal{H}}$, we have $X_1 = X_1^*$ and $X_2=X_2^*$, where $X=X_1 \oplus X_2$ is 
the decomposition of $X$ in $\widetilde{\mathcal{H}}$.
\vspace{0.1in}

(vii) (a) Let $[0,1] \ni \tau \longrightarrow A_{\tau} \in \mathcal{B}_{s.a}(\mathcal{H})$( set of bounded self-adjoint operators in $\mathcal{H}$) be
continuous in operator norm , and let $\widetilde{\mathcal{H}} \ni X\equiv X_{1\tau} \oplus X_{2\tau}$ be the self-adjoint decomposition with respect to $A_{\tau}.$
Then $\tau \longrightarrow X_{1\tau} , ~X_{2\tau} \in \widetilde{\mathcal{H}}$ are continuous.
\vspace{0.1in}

(b) Let $\{A_{\tau}\}_{\tau \in [0,1]}$ be a family of unbounded self-adjoint operators in $\mathcal{H}$ such that $[0,1] \ni \tau \longrightarrow \left(A_{\tau} + \textup{i}\right)^{-1}$
is continuous in operator norm. Then the conclusions of $(vii) (a)$ is valid for the decomposition of $\widetilde{\mathcal{H}}$ with respect to
$B_{\tau} \equiv \left(A_{\tau} + \textup{i}\right)^{-1}$.
\end{lmma}
\begin{proof}
The proofs of $(i)$ to $(iii)$ are standard and for $(iv)$, we note that
since $\textup{Ker}\left(\mathcal{M}_{B}\right) = \textup{Ker}\left(\mathcal{M}_{B^*}\right)$,
$X\in \textup{Ker}\left(\mathcal{M}_{B}\right)$ if and only if $ X^*\in \textup{Ker}\left(\mathcal{M}_{B}\right)$ and hence for any $X\in \textup{Ker}\left(\mathcal{M}_{B}\right)$
can be written as $ X = \left(\frac{X+X^*}{2}\right) + \textup{i} \left(\frac{X-X^*}{2\textup{i}}\right)$, proving that $\textup{Ker}\left(\mathcal{M}_{B}\right)$
is generated by its self-adjoint elements. Similarly, by a similar argument we conclude that $ \overline{\textup{Ran}\left(\mathcal{M}_{B}\right)}$ is also generated
by its self-adjoint elements.

Let  $X\in \widetilde{\mathcal{H}}$, and  $X = X_1 \oplus X_2$ and $X^* = \left(X^*\right)_1 \oplus \left(X^*\right)_2$ be the corresponding decompositions of $X$ and $X^*$ in $\widetilde{\mathcal{H}}$.
Then for any $Y_1=Y_1^*\in \textup{Ker}\left(\mathcal{M}_{B}\right)$,
$$\langle X,Y_1\rangle _2 = \langle X_1,Y_1\rangle _2 = \textup{Tr}\{X_1^*Y_1\} = \textup{Tr}\{Y_1X_1^*\} = \langle Y_1,X_1^*\rangle _2 = \overline{\langle X_1^*,Y_1\rangle _2}.$$
But on the other hand,
$$\langle X,Y_1\rangle _2 = \textup{Tr}\{X^*Y_1\} = \textup{Tr}\{(Y_1X)^*\} = \overline{\textup{Tr}\{Y_1X\}} = \overline{\langle X^*,Y_1\rangle _2} = \overline{\langle (X^*)_1,Y_1\rangle _2}$$ and
hence $ \langle (X^*)_1- X_1^*,Y_1\rangle _2 = 0$ $\forall~~ Y_1=Y_1^*\in \textup{Ker}\left(\mathcal{M}_{B}\right)$, which implies that 
\vspace{0.1in}

$ \hspace{-0.6cm} \langle (X^*)_1- X_1^*,Y \rangle _2 = 0 ~~\forall~~ Y \in \textup{Ker}\left(\mathcal{M}_{B}\right)$, proving that $(X^*)_1 = X_1^*$.
Similarly, by the same argument we conclude that $\left(X^*\right)_2 = X_2^*$.
\vspace{0.1in}

The result $(v)$ and $(vi(a))$ follows from $(iv)$ and $(v)$ respectively. For $(vi(b))$, it suffices to note that any $X\in \mathcal{B}_2(\mathcal{H})$ commuting
with $(A+\textup{i})^{-1}$ commutes with the spectral family $E_{A}(.)$ of $A$.

For $(vii(a))$,  since the map $[0,1]\ni \tau \longrightarrow \mathcal{M}_{A_{\tau}}$ is holomorphic, then ( using {\bf{Theorem 1.8}}, ~page 370,~\cite{kato}~) we 
conclude that the map $[0,1]\ni \tau \longrightarrow P_0(\tau)$ (where $P_0(\tau)$ is the projection onto $\textup{Ker}(\mathcal{M}_{A_{\tau}})$) is continuous 
and since $X_{1\tau} \equiv P_0(\tau) X$ we get that the map $[0,1]\ni \tau \longrightarrow X_{1\tau}$ is continuous. Similarly, since the map $[0,1]\ni \tau \longrightarrow I - P_0(\tau)$
is continuous and $X_{2\tau} = (I - P_0(\tau))X$ then we conclude that the map $[0,1]\ni \tau \longrightarrow X_{2\tau}$ is also continuous.

Conclusions of $(vii(b))$ follows immediately  from $(vii(a))$ since the map $[0,1] \ni \tau \longrightarrow  \mathcal{M}_{\left(A_{\tau} + \textup{i}\right)^{-1}}$
is holomorphic, and since $\mathcal{M}_{\left(A_{\tau} + \textup{i}\right)^{-1}}$ is normal for each $\tau$.
~~~~~~~~~~~~~~~~~~~~~~~~~~~~~~~~~~~~~~~~~~~~~~~~~~~~~~~~~~~~~~~~~~~~~~~~~~~~~~~~~~~~~~~~~~~~~~~~~~~~~~~~~~~~~~~~~~~~~~~~~~~~~~~~~~~~~~~~~~~~~~~\end{proof}

\begin{rmrk}\label{rmk: b2rmk}
Let $A$ and $V$ be two bounded self-adjoint operators in an infinite dimensional Hilbert space $\mathcal{H}$ such that $V\in\mathcal{B}_2(\mathcal{H})$
and $A_{\tau} = A + \tau V$ ($0\leq \tau \leq1$). Apply Lemma \ref{lmma: decomp} with $B = A$ and $A_{\tau}$ respectively to get $V = V_1 \oplus V_2 = V_{1\tau} \oplus V_{2\tau},$
with $V_j$ and $V_{j\tau}~ (j=1,2)$ self-adjoint and therefore $\|V\|_2^2 = \|V_1\|_2^2 + \|V_2\|_2^2 = \|V_{1\tau}\|_2^2 + \|V_{2\tau}\|_2^2 ~~\forall~ 0\leq \tau \leq1$.
\end{rmrk}

\begin{thm}\label{thm: mainbddthm}
Let $A$ and $V$ be two bounded self-adjoint operators in an infinite dimensional Hilbert space $\mathcal{H}$ such that $V\in\mathcal{B}_2(\mathcal{H})$.
Then there exist a unique real-valued function $\eta \in L^1([a,b])$ such that
\begin{equation}\label{eq: bddeq} 
\textup{Tr} \left[ p(A+V) - p(A) - D^{(1)}p(A)(V) - \frac{1}{2}D^{(2)}p(A)(V,V)\right] = \int_a^b p'''(\lambda)\eta(\lambda)d\lambda,
\end{equation}
where $p(.)$ is a polynomial in $[a,b], ~~a=\left[\inf \sigma(A)\right] - \|V\|,~~b= \left[\sup \sigma(A)\right] + \|V\|$ and $\itt_a^b\eta(\lambda)d\lambda = \frac{1}{6} Tr(V^3).$
\end{thm}
\begin{proof}
It will be sufficient to prove the theorem for $p(\lambda) = \lambda^r ~~(r\geq 0)$. Note that for $r=0,1$ or $2$, both sides of \eqref{eq: bddeq} are
identically zero. We set $A_{\tau} = A + \tau V$ and $0\leq \tau \leq 1$. Then by lemma \ref{lmma: bddexplmma}, we have that 
\begin{equation*}
\begin{split}
 \hspace{-3cm} \textup{Tr} \left[(A+V)^r - A^r -D^{(1)}(A^r)(V) - \frac{1}{2}D^{(2)}(A^r)(V,V)\right]\\
&  \hspace{-8cm} = r \sum_{k=0}^{r-2} ~\int_0^1ds\int_0^sd\tau ~~\textup{Tr}\left[VA_{\tau}^{r-k-2}VA_{\tau}^k - VA^{r-k-2}VA^k\right]\\
\end{split}
\end{equation*}
\begin{equation}\label{eq: firstbddexpre}
\begin{split}
\hspace{1cm}  = r(r-1) \int_0^1ds\int_0^sd\tau ~\textup{Tr}\left[V_{1\tau}^2A_{\tau}^{r-2} - V_1^2A^{r-2}\right]\\
& \hspace{-7cm} + ~r\sum_{k=0}^{r-2} ~\int_0^1ds\int_0^sd\tau~\textup{Tr}\left[V_{2\tau}A_{\tau}^{r-k-2}V_{2\tau}A_{\tau}^k - V_{2}A^{r-k-2}V_{2}A^k\right],
\end{split}
\end{equation}
where we have also noted the invariance, orthogonality and continuity properties in  Lemma \ref{lmma: decomp} $(ii) - (vii)$ and set $V = V_1 \oplus V_2 = V_{1\tau} \oplus V_{2\tau} \in \mathcal{B}_2(\mathcal{H})$ as in Remark \ref{rmk: b2rmk}.
Using the spectral families $E_{\tau}(.)$ and $E(.)$ of the self-adjoint operators $A_{\tau}$ and $A$ respectively and integrating by-parts, the first term of the expression \eqref{eq: firstbddexpre}
is equal to
\begin{equation*}
\begin{split}
r(r-1) \int_0^1ds\int_0^sd\tau \int_a^b \lambda^{r-2} ~\textup{Tr}\left[V_{1\tau}^2E_{\tau}(d\lambda) - V_1^2E(d\lambda)\right]\\   
& \hspace{-9cm} = r(r-1) \int_0^1ds\int_0^sd\tau \{ \lambda^{r-2} ~\textup{Tr}\left[V_{1\tau}^2E_{\tau}(\lambda) - V_1^2E(\lambda)\right]\mid_{\lambda=a}^{b}\\
& \hspace{-3cm} - \int_a^b (r-2)\lambda^{r-3} ~\textup{Tr}\left[V_{1\tau}^2E_{\tau}(\lambda) - V_1^2E(\lambda)\right]d\lambda \}\\
\end{split}
\end{equation*}
\begin{equation}\label{eq: firstfinalexp}
\begin{split}
& \hspace{0cm} = r(r-1) b^{r-2} \int_0^1ds\int_0^sd\tau ~\textup{Tr}\left[V_{1\tau}^2 - V_1^2 \right]\\
& \hspace{1cm} + r(r-1)(r-2)\int_0^1 ds\int_0^s d\tau \int_a^b \lambda^{r-3} ~\textup{Tr}\left[V_1^2E(\lambda) - V_{1\tau}^2E_{\tau}(\lambda) \right]d\lambda.
\end{split}
\end{equation}
Since $V_2 \in \overline{\textup{Ran}(\mathcal{M}_{A})}$, then there exists a sequence $\{V_2^{(n)}\} \subseteq \textup{Ran}(\mathcal{M}_{A})$ such that

\hspace{-0.8cm} $\|V_2^{(n)} - V_2\|_2 \longrightarrow 0$ as $n \longrightarrow \infty$ and $V_2^{(n)} = AY_0^{(n)} - Y_0^{(n)}A$, for a sequence $\{Y_0^{(n)}\} \subseteq \mathcal{B}_2(\mathcal{H})$.
Similarly, for every $\tau \in (0,1]$, there exists a sequence $\{V_{2\tau}^{(n)}\} \subseteq \textup{Ran}(\mathcal{M}_{A_{\tau}})$ such that 
$\|V_{2\tau} ^{(n)} - V_{2\tau}\|_2 \longrightarrow 0$ point-wise as $n \longrightarrow \infty$ and $V_{2\tau}^{(n)} = A_{\tau}Y^{(n)} - Y^{(n)}A_{\tau}$, 
for some sequence $\{Y^{(n)}\} \subseteq \mathcal{B}_2(\mathcal{H})$. 
Observe that $Y_0^{(n)}$ and $Y^{(n)}$ must be skew-adjoint for each $n$, since $V_2^{(n)}$ and $V_{2\tau}^{(n)}$ can be chosen to be self-adjoint.
Furthermore, by lemma \ref{lmma: decomp} $(vii)(a)$, the map $[0,1]\ni \tau \longrightarrow V_{1\tau} , V_{2\tau} $ are continuous. 
\vspace{0.1in}

Hence the second term of the expression \eqref{eq: firstbddexpre} is equal to 
\begin{equation*}
\begin{split}
r \int_0^1ds\int_0^sd\tau~\lim_{n\rightarrow \infty}~\sum_{k=0}^{r-2}~\textup{Tr}\{V_{2\tau}A_{\tau}^{r-k-2}V_{2\tau}^{(n)}A_{\tau}^k - V_{2}A^{r-k-2}V_{2}^{(n)}A^k\}\\
& \hspace{-12cm} =  r \int_0^1ds\int_0^sd\tau\lim_{n\rightarrow \infty}~\sum_{k=0}^{r-2}\int_a^b\int_a^b \lambda^{r-k-2}\mu^{k} ~\textup{Tr}\{V_{2\tau}E_{\tau}(d\lambda)V_{2\tau}^{(n)}E_{\tau}(d\mu) - V_{2}E(d\lambda)V_{2}^{(n)}E(d\mu)\}\\
& \hspace{-12cm} =  r \int_0^1ds\int_0^sd\tau ~\lim_{n\rightarrow \infty}~ \int_a^b\int_a^b \phi(\lambda,\mu) ~\textup{Tr}\{V_{2\tau}E_{\tau}(d\lambda)V_{2\tau}^{(n)}E_{\tau}(d\mu) - V_{2}E(d\lambda)V_{2}^{(n)}E(d\mu)\},
\end{split}
\end{equation*}
where $\phi(\lambda,\mu) = \frac{\lambda^{r-1} - \mu^{r-1}}{\lambda - \mu}$ if $\lambda \neq \mu$ ~;~ $= (r-1) \lambda^{r-2}$ if $\lambda = \mu$, and where
the interchange of the limit and the integration is justified by an application of the bounded convergence theorem. Furthermore using the representation 
of $V_{2\tau}^{(n)} \in \textup{Ran}(\mathcal{M}_{A_{\tau}})$, the above reduces to
\begin{equation*}
\begin{split}
r \int_0^1ds\int_0^sd\tau ~\lim_{n\rightarrow \infty}~ \int_a^b\int_a^b \phi(\lambda,\mu) ~\textup{Tr}\{V_{2\tau}E_{\tau}(d\lambda)\left[ A_{\tau}Y^{(n)} - Y^{(n)}A_{\tau}\right]E_{\tau}(d\mu) \\ 
&\hspace{-5cm} - V_{2}E(d\lambda)\left[ AY_0^{(n)} - Y_0^{(n)}A \right]E(d\mu)\}\\
& \hspace{-14cm} = r \int_0^1ds\int_0^sd\tau \lim_{n\rightarrow \infty} \int_a^b\int_a^b \left(\lambda^{r-1} - \mu^{r-1}\right) \textup{Tr}\{V_{2\tau}E_{\tau}(d\lambda)Y^{(n)}E_{\tau}(d\mu) -  V_{2}E(d\lambda) Y_0^{(n)}E(d\mu)\}\\
\end{split}
\end{equation*}
\begin{equation}\label{eq: intbyparts}
\hspace{-3.2cm} = r \int_0^1ds\int_0^sd\tau~ \lim_{n\rightarrow \infty}~ \int_a^b \lambda^{r-1} ~\textup{Tr}\{V_{2\tau}\left[E_{\tau}(d\lambda),Y^{(n)}\right] -  V_{2}\left[E(d\lambda), Y_0^{(n)}\right]\}. 
\end{equation}
Again by twice integrating by-parts, the expression in \eqref{eq: intbyparts} is equal to
\begin{equation*}
\begin{split}
r \int_0^1ds\int_0^sd\tau~ \lim_{n\rightarrow \infty}~ \{ \lambda^{r-1} ~\textup{Tr}\left(V_{2\tau}\left[E_{\tau}(\lambda),Y^{(n)}\right] -  V_{2}\left[E(\lambda), Y_0^{(n)}\right]\right)|_{\lambda = a}^b\\
& \hspace{-8cm} - \int_a^b (r-1)\lambda^{r-2} ~\textup{Tr}\left(V_{2\tau}\left[E_{\tau}(\lambda),Y^{(n)}\right] -  V_{2}\left[E(\lambda), Y_0^{(n)}\right]\right) d\lambda\}\\
& \hspace{-13cm} = -r(r-1) \int_0^1ds\int_0^sd\tau~ \lim_{n\rightarrow \infty}~\int_a^b \lambda^{r-2} ~\textup{Tr}\{V_{2\tau}\left[E_{\tau}(\lambda),Y^{(n)}\right] -  V_{2}\left[E(\lambda), Y_0^{(n)}\right]\}~ d\lambda
\end{split}
\end{equation*}
\begin{equation}\label{eq: longexpre}
\begin{split}
= -r(r-1) \int_0^1ds\int_0^sd\tau~ \lim_{n\rightarrow \infty}~ \{ \lambda^{r-2} \left(\int_a^{\lambda}~\textup{Tr}\left(V_{2\tau}\left[E_{\tau}(\mu),Y^{(n)}\right] -  V_{2}\left[E(\mu), Y_0^{(n)}\right]\right) d\mu\right)\}|_{\lambda = a}^b\\
& \hspace{-16.8cm}  + r(r-1) \int_0^1ds\int_0^sd\tau~ \lim_{n\rightarrow \infty} \int_a^b (r-2) \lambda^{r-3} \left(\int_{a}^{\lambda}~\textup{Tr}\left(V_{2\tau}\left[E_{\tau}(\mu),Y^{(n)}\right] -  V_{2}\left[E(\mu), Y_0^{(n)}\right]\right) d\mu \right)d\lambda\\
& \hspace{-16.8cm} = -r(r-1) b^{r-2} \int_0^1ds\int_0^sd\tau~ \lim_{n\rightarrow \infty}~ \int_a^{b}~\textup{Tr}\left(V_{2\tau}\left[E_{\tau}(\mu),Y^{(n)}\right] -  V_{2}\left[E(\mu), Y_0^{(n)}\right]\right) d\mu\\
& \hspace{-16.8cm}  + r(r-1)(r-2) \int_0^1ds\int_0^sd\tau\lim_{n\rightarrow \infty}\int_a^b\lambda^{r-3}\left(\int_{a}^{\lambda}\textup{Tr}\left(V_{2\tau}\left[E_{\tau}(\mu),Y^{(n)}\right] -  V_{2}\left[E(\mu), Y_0^{(n)}\right]\right) d\mu \right)d\lambda.
\end{split}
\end{equation}
Next we note that by an integration by-parts, 
\begin{equation*}
\begin{split}
\textup{Tr}\left( V_{2\tau}^2 - V_2^2 \right) = \lim_{n\rightarrow \infty} \textup{Tr}\left(V_{2\tau}V_{2\tau}^{(n)} - V_2 V_2^{(n)}\right) = \lim_{n\rightarrow \infty} \textup{Tr}\left(V_{2\tau}\left[A_{\tau},Y^{(n)}\right] - V_2\left[A,Y_0^{(n)}\right]\right)\\
& \hspace{-15.5cm} =  \lim_{n\rightarrow \infty} \int_a^b\mu \textup{Tr}\left(V_{2\tau}\left[E_{\tau}(d\mu),Y^{(n)}\right] - V_2\left[E(d\mu),Y_0^{(n)}\right]\right)\\
& \hspace{-15.5cm} = \lim_{n\rightarrow \infty} \left[\mu \textup{Tr}\left(V_{2\tau}\left[E_{\tau}(\mu),Y^{(n)}\right] - V_2\left[E(\mu),Y_0^{(n)}\right]\right)|_{\mu =a}^b - \int_a^b \textup{Tr}\left(V_{2\tau}\left[E_{\tau}(\mu),Y^{(n)}\right] - V_2\left[E(\mu),Y_0^{(n)}\right]\right)d\mu\right].\\
\end{split}
\end{equation*}
The boundary term above vanishes and substituting the above in the first expression in \eqref{eq: longexpre}, we get that the right hand side of \eqref{eq: longexpre}
\begin{equation*}
\begin{split}
 = r(r-1) b^{r-2} \int_0^1ds\int_0^sd\tau~\textup{Tr}\left( V_{2\tau}^2 - V_2^2 \right) + r(r-1)(r-2) \int_0^1ds\int_0^sd\tau\lim_{n\rightarrow \infty}\int_a^b\lambda^{r-3}\eta_{2\tau}^{(n)}(\lambda)d\lambda,
\end{split}
\end{equation*}
where $\eta_{2\tau}^{(n)}(\lambda) = \itt_{a}^{\lambda}\textup{Tr}\left(V_{2\tau}\left[E_{\tau}(\mu),Y^{(n)}\right] -  V_{2}\left[E(\mu), Y_0^{(n)}\right]\right) d\mu.$

\hspace{-0.8cm} Hence
\begin{equation}\label{eq: secondfinalexp}
\begin{split}
& r\sum_{k=0}^{r-2} \int_0^1ds\int_0^sd\tau~\textup{Tr}\left[V_{2\tau}A_{\tau}^{r-k-2}V_{2\tau}A_{\tau}^k - V_{2}A^{r-k-2}V_{2}A^k\right]\\
& = r(r-1) b^{r-2} \int_0^1ds\int_0^sd\tau~\textup{Tr}\left( V_{2\tau}^2 - V_2^2 \right) + r(r-1)(r-2) \int_0^1ds\int_0^sd\tau\lim_{n\rightarrow \infty}\int_a^b\lambda^{r-3}\eta_{2\tau}^{(n)}(\lambda)d\lambda.
\end{split}
\end{equation}
Combining \eqref{eq: firstfinalexp} and \eqref{eq: secondfinalexp} and since $\|V\|_2^2 = \textup{Tr}\left(V_{1\tau}^2 + V_{2\tau}^2\right) = \textup{Tr}\left(V_1^2 + V_2^2\right)$,
we conclude that

\begin{equation*}
\begin{split}
\textup{Tr} \left[(A+V)^r - A^r -D^{(1)}(A^r)(V) - \frac{1}{2}D^{(2)}(A^r)(V,V)\right]\\
& \hspace{-8cm} = r(r-1)(r-2)\int_0^1 ds\int_0^s d\tau \int_a^b \lambda^{r-3} ~\textup{Tr}\left[V_1^2E(\lambda) - V_{1\tau}^2E_{\tau}(\lambda) \right]d\lambda\\
& \hspace{-7cm} + r(r-1)(r-2) \int_0^1ds\int_0^sd\tau\lim_{n\rightarrow \infty}\int_a^b\lambda^{r-3}~\eta_{2\tau}^{(n)}(\lambda)d\lambda\\
& \hspace{-8cm} = r(r-1)(r-2) \int_0^1ds\int_0^sd\tau\lim_{n\rightarrow \infty}\int_a^b\lambda^{r-3}~\eta_{\tau}^{(n)}(\lambda)d\lambda\\
& \hspace{-8cm} = \lim_{n\rightarrow \infty}\int_a^b\left(\lambda^r\right)''' ~\eta^{(n)}(\lambda)d\lambda~,\quad \text{where} \quad
\end{split}
\end{equation*}
$\eta^{(n)}(\lambda) \equiv \itt_0^1ds\itt_0^s d\tau ~\eta_{\tau}^{(n)}(\lambda)$ and $ \eta_{\tau}^{(n)} (\lambda) = \left[~\textup{Tr}\{V_1^2E(\lambda) - V_{1\tau}^2E_{\tau}(\lambda) \}  + \eta_{2\tau}^{(n)}(\lambda)\right]$, 
the interchange of limit and the $\tau$- and $s$- integral is justified by an easy application of bounded convergence theorem. Note that $\eta^{(n)}$ is a
real-valued function $\forall ~n$.
\vspace{0.1in}

Next we want to show that $\{\eta^{(n)}\}$ is cauchy in $L^1([a,b])$ and we follow the idea from \cite{gesztesykop}. For that let $f\in L^{\infty}([a,b])$. Define $g(\lambda) = \itt_{a}^{\lambda} f(t)dt$ and $h(\lambda) = \itt_a^{\lambda} g(\mu)d\mu$,
then $g'(\lambda) = f(\lambda)$ a.e. and $h'(\lambda) = g(\lambda)$. Now consider the expression
\begin{equation*}
\begin{split}
\int_a^b f(\lambda) \left[\eta_{\tau}^{(n)}(\lambda) - \eta_{\tau}^{(m)}(\lambda)\right] d\lambda = \int_a^b f(\lambda) \left[\eta_{2\tau}^{(n)}(\lambda) - \eta_{2\tau}^{(m)}(\lambda)\right] d\lambda\\
& \hspace{-10cm} = \int_a^b h''(\lambda) d\lambda \left( \int_{a}^{\lambda}\textup{Tr}\left(V_{2\tau}\left[E_{\tau}(\mu),Y^{(n)}-Y^{(m)}\right] -  V_{2}\left[E(\mu), Y_0^{(n)}-Y_0^{(m)}\right] \right)d\mu\right),
\end{split}
\end{equation*}
which on integration by-parts twice and on observing that the boundary term for $\lambda =a$ vanishes, leads to
\begin{equation}\label{eq: bddintbypart}
\begin{split}
h'(b) \int_{a}^{b}\textup{Tr}\left(V_{2\tau}\left[E_{\tau}(\mu),Y^{(n)}-Y^{(m)}\right] -  V_{2}\left[E(\mu), Y_0^{(n)}-Y_0^{(m)}\right] \right)d\mu\\
& \hspace{-10cm} - \{h(\lambda) \textup{Tr}\left(V_{2\tau}\left[E_{\tau}(\lambda),Y^{(n)}-Y^{(m)}\right] -  V_{2}\left[E(\lambda), Y_0^{(n)}-Y_0^{(m)}\right] \right)|_{\lambda =a}^b\\
& \hspace{-8cm}  -\int_a^b h(\lambda) \textup{Tr}\left(V_{2\tau}\left[E_{\tau}(d\lambda),Y^{(n)}-Y^{(m)}\right] -  V_{2}\left[E(d\lambda), Y_0^{(n)}-Y_0^{(m)}\right] \right)\}\\
& \hspace{-12cm} = h'(b) \int_{a}^{b}\textup{Tr}\left(V_{2\tau}\left[E_{\tau}(\mu),Y^{(n)}-Y^{(m)}\right] -  V_{2}\left[E(\mu), Y_0^{(n)}-Y_0^{(m)}\right] \right)d\mu\\
& \hspace{-9cm} +\int_a^b h(\lambda) \textup{Tr}\left(V_{2\tau}\left[E_{\tau}(d\lambda),Y^{(n)}-Y^{(m)}\right] -  V_{2}\left[E(d\lambda), Y_0^{(n)}-Y_0^{(m)}\right] \right).
\end{split}
\end{equation}
Next we use the identity $$\textup{Tr}\left(V_{2\tau}V_{2\tau}^{(n)} - V_2 V_2^{(n)}\right) = - \int_a^b \textup{Tr}\left(V_{2\tau}\left[E_{\tau}(\mu),Y^{(n)}\right] - V_2\left[E(\mu),Y_0^{(n)}\right]\right)d\mu$$
to reduce the the above expression in \eqref{eq: bddintbypart} to
\begin{equation}\label{eq: comb1}                           
g(b) \textup{Tr}\left(V_2\left[V_2^{(n)}-V_2^{(m)}\right]-V_{2\tau}\left[V_{2\tau}^{(n)}-V_{2\tau}^{(m)}\right]\right) + \textup{Tr}\left(V_{2\tau}\left[h(A_{\tau}),Y^{(n)}-Y^{(m)}\right]-V_2\left[h(A),Y_0^{(n)}-Y_0^{(m)}\right]\right).
\end{equation}
But on the other hand,
\begin{equation*}
\left[h(A),Y^{(n)}\right] = \int_a^b\int_a^b\frac{h(\lambda)-h(\mu)}{\lambda-\mu} ~~E(d\lambda)V_2^{(n)}E(d\mu), \quad \text{and} \quad
\end{equation*}
\begin{equation}\label{eq: comb2}
\textup{Tr}\left( V_2\left[h(A),Y_0^{(n)}-Y_0^{(m)}\right]\right) = \int_a^b\int_a^b\frac{h(\lambda)-h(\mu)}{\lambda-\mu} ~\textup{Tr}\left(V_2E(d\lambda)\left[V_2^{(n)}-V_2^{(m)}\right]E(d\mu)\right)
\end{equation}
and hence as in Birman-Solomyak (\cite{birmansolomyak},\cite{birsolomyak}) and in \cite{chttosinha} ,
$$\left|\textup{Tr}\left( V_2\left[h(A),Y_0^{(n)}-Y_0^{(m)}\right]\right)\right| \leq \|h\|_{Lip} \|V_2\|_2\left\|\left[V_2^{(n)}-V_2^{(m)}\right]\right\|_2 \leq (b-a) \|f\|_{\infty}\|V\|_2\left\|\left[V_2^{(n)}-V_2^{(m)}\right]\right\|_2$$
and hence 
$$\frac{\left|\int_a^b f(\lambda) \left[\eta_{\tau}^{(n)}(\lambda) - \eta_{\tau}^{(m)}(\lambda)\right] d\lambda\right|}{\|f\|_{\infty}} \leq 2(b-a)\|V\|_2\left( \left\|\left[V_2^{(n)}-V_2^{(m)}\right]\right\|_2 + \left\|\left[V_{2\tau}^{(n)}-V_{2\tau}^{(m)}\right]\right\|_2\right)$$
$$i.e.~~~~~ \|\eta_{\tau}^{(n)} - \eta_{\tau}^{(m)} \|_{L^1} ~~\leq 2(b-a)\|V\|_2\left( \left\|\left[V_2^{(n)}-V_2^{(m)}\right]\right\|_2 + \left\|\left[V_{2\tau}^{(n)}-V_{2\tau}^{(m)}\right]\right\|_2\right), $$
which converges to $0$ as $n,m\rightarrow \infty$ and $\forall ~~\tau\in[0,1].$ A similar computation also shows that $\|\eta_{\tau}^{(n)}\|_{L^1} \leq 2(b-a)\|V\|_2^2$.
Therefore $\{\itt_0^1ds\itt_0^s d\tau ~\eta_{\tau}^{(n)}(\lambda) \equiv \eta^{(n)}(\lambda)\}$ is also cauchy in $L^1([a,b])$
and thus there exists a function $\eta \in L^1([a,b])$ such that $\|\eta^{(n)} - \eta\|_{L^1}\rightarrow 0$ as $n\longrightarrow \infty$, by the bounded
convergence theorem and hence also $\|\eta\|_{L^1}\leq (b-a)\|V\|_2^2$. Therefore,
$$ \lim_{n\longrightarrow \infty}\int_a^b\lambda^{r-3}\eta^{(n)}(\lambda)d\lambda = \int_a^b\lambda^{r-3}\eta(\lambda)d\lambda\quad \text{and hence} \quad$$
$$\textup{Tr} \left[(A+V)^r - A^r -D^{(1)}(A^r)(V) - \frac{1}{2}D^{(2)}(A^r)(V,V)\right] = r(r-1)(r-2)\int_a^b\lambda^{r-3}\eta(\lambda)d\lambda. $$
For uniqueness, let us assume that there exists $\eta_1,~\eta_2\in L^1([a,b])$ such that 
$$\textup{Tr} \left[ p(A+V) - p(A) - D^{(1)}p(A)(V) - \frac{1}{2}D^{(2)}p(A)(V,V)\right] = \int_a^b p'''(\lambda)\eta_j(\lambda)d\lambda,$$
where $p(.)$ is a polynomial and $j=1,2$. Therefore $$\int_a^b p'''(\lambda)~\eta(\lambda)~d\lambda = 0 ~~\forall \quad \text{polynomials} \quad p(.) \quad \text{and} \quad \eta \equiv \eta_1 -\eta_2\in L^1([a,b]),$$
which together with the fact that $\int_a^b\eta_1(\lambda)~d\lambda = \int_a^b\eta_2(\lambda)~d\lambda = \frac{1}{6}\textup{Tr}(V^3)$~(which one can easily
arrive at by setting $p(\lambda) = \lambda^3$ in the above formula), implies that
$$\int_a^b\lambda^r\eta(\lambda)d\lambda = 0 ~~\forall ~~r\geq 0.\quad \text{ Hence by an application of Fubini's theorem, we get that} \quad$$
$$ \int_{-\infty}^{\infty}~e^{-\textup{i}t\lambda}~\eta(\lambda)~ d\lambda ~= ~~\sum_{n=0}^{\infty} ~\frac{1}{n!}~\int_{-\infty}^{\infty}~(-\textup{i}t\lambda)^n~\eta(\lambda)~d\lambda = 0.$$ Hence $$\int_{-\infty}^{\infty}~\textup{e}^{-\textup{i}t\lambda}~\eta(\lambda)~ d\lambda = 0 ~\forall ~t\in \mathbb{R}.$$
Therefore $\eta$ is an $L^1([a,b])$- function whose Fourier transform $\hat{\eta}(t)$ vanishes identically, implying that $\eta = 0$ or $\eta_1 = \eta_2$ a.e.
~~~~~~~~~~~~~~~~~~~~~~~~~~~~~~~~~~~~~~~~~~~~~~~~~~~~~~~~~~~~~~~~~~~~~~~~~~~~~~~~\end{proof}
\begin{crlre}
Let $A$ and $V$ be two bounded self-adjoint operators in an infinite dimensional Hilbert space $\mathcal{H}$ such that $V\in\mathcal{B}_2(\mathcal{H})$.
Then the function $\eta \in L^1([a,b])$ obtained as in theorem \ref{thm: mainbddthm} satisfies the following equation
\begin{equation*}
\begin{split}
\int_{a}^{b} f(\lambda) \eta(\lambda)d\lambda = \int_0^1ds\int_0^sd\tau \int_a^{b}\int_a^{b}\frac{h(\lambda)-h(\mu)}{\lambda -\mu} ~\textup{Tr}\left[VE_{\tau}(d\lambda)VE_{\tau}(d\mu) - VE(d\lambda)VE(d\mu) \right],
\end{split}
\end{equation*}
where $f(\lambda)$ , $g(\lambda)$ and $h(\lambda)$  are as in the proof of the theorem \ref{thm: mainbddthm}.
\end{crlre}
\begin{proof}
By Fubini's theorem we have that 
\begin{equation*}
\begin{split} 
\int_{a}^{b} f(\lambda) \eta ^{(n)}(\lambda) d\lambda = \int_0^1ds \int_0^sd\tau \int_{a}^{b} f(\lambda) \left[~\textup{Tr}\{V_1^2E(\lambda) - V_{1\tau}^2E_{\tau}(\lambda) \}  + \eta_{2\tau}^{(n)}(\lambda)\right] d\lambda.
\end{split}
\end{equation*}
But 
\begin{equation*}
\begin{split} 
\int_{a}^{b} f(\lambda) ~\textup{Tr}\{V_1^2E(\lambda) - V_{1\tau}^2E_{\tau}(\lambda) \} d\lambda = \int_{a}^{b} g'(\lambda)  \textup{Tr}\left[ V_1^2E(\lambda)-V_{1\tau}^2E_{\tau}(\lambda)\right] d\lambda,
\end{split}
\end{equation*}
which by integrating  by-parts leads to
\begin{equation}\label{eq: bddcor1}
\begin{split} 
g(b) ~\textup{Tr}\left[ V_{1}^2 - V_{1\tau}^2 \right] + \int_{a}^{b} g(\lambda)  \textup{Tr}\left[ V_{1\tau}^2E_{\tau}(d\lambda)-V_{1}^2E(d\lambda)\right] \\
& \hspace{-10cm} = g(b) ~\textup{Tr}\left[ V_{1}^2 - V_{1\tau}^2 \right] + \textup{Tr}\left[ V_{1\tau}^2 h'(A_{\tau})-V_{1}^2 h'(A)\right]\\
& \hspace{-10cm} =  g(b) ~\textup{Tr}\left[ V_{1}^2 - V_{1\tau}^2 \right] + \int_a^{b}\int_a^{b}\frac{h(\lambda)-h(\mu)}{\lambda -\mu} ~\textup{Tr}\left[V_{1\tau}E_{\tau}(d\lambda)V_{1\tau}E_{\tau}(d\mu) -V_1E(d\lambda)V_1E(d\mu)\right].
\end{split}
\end{equation}
Again by repeating the same above calculations to get \eqref{eq: comb1} and \eqref{eq: comb2} as in the proof of the theorem \ref{thm: mainbddthm}, we conclude that
\begin{equation}\label{eq: bddcor2}
\begin{split} 
\int_{a}^{b} f(\lambda) \eta_{2\tau}^{(n)}(\lambda)  d\lambda = g(b) ~\textup{Tr}\left[ V_2 V_2^{(n)} - V_{2\tau} V_{2\tau}^{(n)} \right]\\
& \hspace{-6cm}  + \int_a^{b}\int_a^{b}\frac{h(\lambda)-h(\mu)}{\lambda -\mu} ~\textup{Tr}\left[ V_{2\tau}E_{\tau}(d\lambda)V_{2\tau}^{(n)}E_{\tau}(d\mu) - V_2E(d\lambda)V_2^{(n)}E(d\mu)\right].
\end{split}
\end{equation}
Combining \eqref{eq: bddcor1} and \eqref{eq: bddcor2} we have,

\begin{equation}\label{eq: bddcor3}
\begin{split} 
\int_{a}^{b} f(\lambda) \eta^{(n)}(\lambda)  d\lambda \\
& \hspace{-3cm} = \textup{Tr}\left[ \left(V_1^2 +V_2 V_2^{(n)}\right) - \left(V_{2\tau}^2 + V_{2\tau} V_{2\tau}^{(n)}\right) \right]\\
& \hspace{-3cm} + \int_0^1 ds\int_0^s d\tau\int_a^{b}\int_a^{b}\frac{h(\lambda)-h(\mu)}{\lambda -\mu} ~\textup{Tr}\left[ VE_{\tau}(d\lambda)\left(V_{1\tau} \oplus V_{2\tau}^{(n)}\right) E_{\tau}(d\mu) - VE(d\lambda)\left(V_1 \oplus V_2^{(n)}\right)E(d\mu)\right].
\end{split}
\end{equation}
But by definition $V_{2}^{(n)}, V_{2\tau}^{(n)}$ converges to $V_2 , V_{2\tau}$ respectively in $\|.\|_2$ and we have already proved that $\eta^{(n)}$ converges to $\eta$ in $ L^1([a,b]).$
 Hence by taking limit on both sides of \eqref{eq: bddcor3} we get that 
\begin{equation}\label{eq: bddcor4}
\begin{split} 
\int_{a}^{b} f(\lambda) \eta(\lambda)  d\lambda = \int_0^1 ds\int_0^s d\tau\int_a^{b}\int_a^{b}\frac{h(\lambda)-h(\mu)}{\lambda -\mu} ~\textup{Tr}\left[VE_{\tau}(d\lambda)V E_{\tau}(d\mu) - VE(d\lambda)VE(d\mu) \right].
\end{split}
\end{equation}
In the right hand side of  \eqref{eq: bddcor4} we have used the fact that

\hspace{-0.8cm} $Var\left(\mathcal{G}_2^{(n)} -\mathcal{G}_2\right) \leq ~\|V\|_2 \left( \|V_{2\tau}^{(n)} - V_{2\tau}\|_2 + ~\|V_2 - V_2^{(n)}\| \right) \longrightarrow 0$
as $n\longrightarrow \infty$, where 

\hspace{-0.8cm} $\mathcal{G}_2^{(n)} (\Delta \times \delta) = \textup{Tr}\left[ VE_{\tau}(\Delta)\left(V_{1\tau} \oplus V_{2\tau}^{(n)}\right) E_{\tau}(\delta) - VE(\Delta)\left(V_1 \oplus V_2^{(n)}\right)E(\delta)\right] $ and

\hspace{-0.8cm} $\mathcal{G}_2\left(\Delta \times \delta\right) = \textup{Tr}\left[ VE_{\tau}(\Delta)V E_{\tau}(\delta) - VE(\Delta)VE(\delta) \right]$
are complex measures on $\mathbb{R}^2$ and 

\hspace{-0.8cm} $Var\left(\mathcal{G}_2^{(n)} -\mathcal{G}_2\right)$ is the variation of $\left(\mathcal{G}_2^{(n)} -\mathcal{G}_2\right)$ and also noted that $ \|g\|_{\textup{Lip}}\leq (b-a) \|f\|_{\infty}.$
\end{proof}

\section{Unbounded Case}

\begin{thm}
Let $A$ be an unbounded self-adjoint operator in a Hilbert space $\mathcal{H}$ and let $\phi:\mathbb{R}\longrightarrow \mathbb{C}$ be such that 
$\int_{-\infty}^{\infty} |\hat{\phi}(t)|~(1 + |t|)^{3} ~dt < \infty$, where $\hat \phi$ is the Fourier transform of $\phi$. Then $\phi(A),~D^{(1)}\phi(A),$ $~D^{(2)}\phi(A)$ exist and 
$$\left[D^{(1)}\phi(A)\right](X) = \textup{i}\int_{-\infty}^{\infty} \hat{\phi}(t)dt\int_0^td\beta ~\textup{e}^{\textup{i}\beta A}X\textup{e}^{\textup{i}(t-\beta) A} \quad \text{and} \quad$$
\begin{equation*}
\begin{split}
\left[D^{(2)}\phi(A)\right](X,Y) = \textup{i}^2\int_{-\infty}^{\infty} \hat{\phi}(t)dt\int_0^td\beta ~\{\int_0^{\beta}d\nu ~\textup{e}^{\textup{i}\nu A}X\textup{e}^{\textup{i}(\beta-\nu)A}Y\textup{e}^{\textup{i}(t-\beta)A}\\ 
& \hspace{-4cm} + \int_0^{t-\beta}d\nu ~\textup{e}^{\textup{i}\beta A}Y\textup{e}^{\textup{i}\nu A}X\textup{e}^{\textup{i}(t-\beta-\nu)A} \},
\end{split}
\end{equation*}
where $X,Y\in \mathcal{B}_2(\mathcal{H})$.
\end{thm}
\begin{proof}
That $\phi(A)$ and the expressions on the right hand side above exist in $\mathcal{B}(\mathcal{H})$ are consequences of the functional calculus and the 
assumption on $\hat{\phi}$. Next
$$ \phi(A+X) - \phi(A) = \int_{-\infty}^{\infty} \hat{\phi}(t)\left[\textup{e}^{\textup{i}t(A+X)} - \textup{e}^{\textup{i}tA}\right]dt = \int_{-\infty}^{\infty} \hat{\phi}(t)dt\int_0^td\beta~ \textup{e}^{\textup{i}\beta (A+X)}\textup{i}X\textup{e}^{\textup{i}(t-\beta) A}.$$
Therefore
\begin{equation*}
\begin{split}
\phi(A+X) - \phi(A) - \textup{i}\int_{-\infty}^{\infty}\hat{\phi}(t)dt\int_0^td\beta ~e^{\textup{i}\beta A}Xe^{\textup{i}(t-\beta) A} \\
& \hspace{-5cm} =  \textup{i}\int_{-\infty}^{\infty} \hat{\phi}(t)dt\int_0^td\beta~ \left[\textup{e}^{\textup{i}\beta (A+X)}X\textup{e}^{\textup{i}(t-\beta)A} - \textup{e}^{\textup{i}\beta A}X\textup{e}^{\textup{i}(t-\beta)A}\right].
\end{split}
\end{equation*}
Using the interpolation inequality 
$$ \left\| \textup{e}^{\textup{i}\beta (A+X)} - \textup{e}^{\textup{i}\beta A}\right\| \leq 2^{(1-\epsilon)}~(\beta \|X\|)^{\epsilon} ~~(0\leq \epsilon \leq1), \quad \text{we get that} \quad$$
\begin{equation*}
\begin{split}
\left\|\phi(A+X) - \phi(A) - \textup{i}\int_{-\infty}^{\infty} \hat{\phi}(t)dt\int_0^td\beta ~\textup{e}^{\textup{i}\beta A}X\textup{e}^{\textup{i}(t-\beta) A} \right\|\\
& \hspace{-3cm} \leq \left(\frac{2^{(1-\epsilon)}}{\epsilon +1}\right) \|X\|^{\epsilon +1} \int_{-\infty}^{\infty} |\hat{\phi}(t)|~(1 + |t|)^{\epsilon +2} ~dt,
\end{split}
\end{equation*}
 which by virtue of the assumption on $\hat{\phi}$ implies that $D^{(1)}\phi(A)$ exist and that 
$$\left[D^{(1)}\phi(A)\right](X) = \textup{i}\int_{-\infty}^{\infty} \hat{\phi}(t)dt\int_0^td\beta ~\textup{e}^{\textup{i}\beta A}X\textup{e}^{\textup{i}(t-\beta) A}. \quad \text{Similarly for } \quad X,Y\in \mathcal{B}(\mathcal{H}),$$
\begin{equation*}
\begin{split}
\left[D^{(1)}\phi(A+X)\right](Y) - \left[D^{(1)}\phi(A)\right](Y) \\
& \hspace{-4cm} = \textup{i}^2\int_{-\infty}^{\infty} \hat{\phi}(t)dt\int_0^td\beta ~\{\int_0^{\beta}d\nu ~\textup{e}^{\textup{i}\nu (A+X)}X\textup{e}^{\textup{i}(\beta-\nu)A}Y\textup{e}^{\textup{i}(t-\beta)(A+X)}\\ 
& \hspace{1cm} + \int_0^{t-\beta}d\nu ~\textup{e}^{\textup{i}\beta A}Y\textup{e}^{\textup{i}\nu (A+X)}X\textup{e}^{\textup{i}(t-\beta-\nu)A} \}
\end{split}
\end{equation*}
and one can verify as before that
\begin{equation*}
\begin{split}
\|\left[D^{(1)}\phi(A+X)\right](Y) - \left[D^{(1)}\phi(A)\right](Y) \\
& \hspace{-4cm} - \textup{i}^2\int_{-\infty}^{\infty} \hat{\phi}(t)dt\int_0^td\beta ~\{\int_0^{\beta}d\nu ~\textup{e}^{\textup{i}\nu A}X\textup{e}^{\textup{i}(\beta-\nu)A}Y\textup{e}^{\textup{i}(t-\beta)A}\\ 
& \hspace{1cm} + \int_0^{t-\beta}d\nu ~\textup{e}^{\textup{i}\beta A}Y\textup{e}^{\textup{i}\nu A}X\textup{e}^{\textup{i}(t-\beta-\nu)A} \}\|
\end{split}
\end{equation*}
$$ \leq~ K \|X\|^{\epsilon +1} \|Y\| \int_{-\infty}^{\infty} |\hat{\phi}(t)|~(1 + |t|)^{\epsilon +2} ~dt\quad \text{(for some $\epsilon > 0$ and some constant} \quad K\equiv K(\epsilon)),$$
proving the expression for $\left[D^{(2)}\phi(A)\right](X,Y).$
~~~~~~~~~~~~~~~~~~~~~~~~~~~~~~~~~~~~~~~~~~~~~~~~~~~~~~~~~~~~~~~~~~~~~~~~~\end{proof}
\begin{thm}
Let $A$ be an unbounded self-adjoint operator in a Hilbert space $\mathcal{H}$, $V$ be a self-adjoint operator such that $V\in \mathcal{B}_3(\mathcal{H})$ 
and furthermore let $\phi \in \mathcal{S}(\mathbb{R})$ (the Schwartz class of smooth functions of rapid decrease). Then 
$$\phi(A+V) - \phi(A) -\left[D^{(1)}\phi(A)\right](V) - \frac{1}{2}\left[D^{(2)}\phi(A)\right](V,V) \in \mathcal{B}_1(\mathcal{H})\quad \text{and} \quad$$
\begin{equation*}
\begin{split}
\textup{Tr}\left[\phi(A+V) - \phi(A) -\left[D^{(1)}\phi(A)\right](V) - \frac{1}{2}\left[D^{(2)}\phi(A)\right](V,V) \right]\\
& \hspace{-10cm} = \int_{-\infty}^{\infty} \textup{i}^2 ~t\hat{\phi}(t)~dt \int_0^t d\nu \int_0^1 ds \int_0^s d\tau ~\textup{Tr}\left[ V\textup{e}^{\textup{i}(t-\nu)A_{\tau}}V\textup{e}^{\textup{i}\nu A_{\tau}} - V\textup{e}^{\textup{i}(t-\nu)A}V\textup{e}^{\textup{i}\nu A} \right],
\end{split}
\end{equation*}
where $A_{\tau} = A +\tau V$ and $0\leq \tau\leq1$.
\end{thm}
\begin{proof}
Since $ [0,1]\ni s \longrightarrow \textup{e}^{\textup{i}tA_s}$ is $\mathcal{B}_3(\mathcal{H})$-continuously differentiable, uniformly in $t$,

\begin{equation*}
\begin{split}
\phi(A+V) - \phi(A) -\left[D^{(1)}\phi(A)\right](V) = \int_{-\infty}^{\infty} \hat{\phi}(t)~dt \left[\textup{e}^{\textup{i}t(A+v)} - \textup{e}^{\textup{i}tA}\right] - \left[D^{(1)}\phi(A)\right](V)\\
& \hspace{-14cm} = \int_{-\infty}^{\infty} \hat{\phi}(t)~dt\int_0^1ds \frac{d}{ds}(\textup{e}^{\textup{i}tA_s}) - \textup{i}\int_{-\infty}^{\infty} \hat{\phi}(t)~dt\int_0^td\beta ~\textup{e}^{\textup{i}\beta A}V\textup{e}^{\textup{i}(t-\beta) A}\\
& \hspace{-14cm} = \int_{-\infty}^{\infty} \hat{\phi}(t)~dt \left[\int_0^1ds \int_0^td\beta ~\textup{e}^{\textup{i}\beta A_s}\textup{i}V\textup{e}^{\textup{i}(t-\beta) A_s} - \textup{i}\int_0^1ds \int_0^td\beta ~\textup{e}^{\textup{i}\beta A}V\textup{e}^{\textup{i}(t-\beta)A} \right]\\
\end{split}
\end{equation*}

\begin{equation}
\begin{split}
 \hspace{1cm} = \int_{-\infty}^{\infty} \textup{i} \hat{\phi}(t)~dt \int_0^1ds \int_0^td\beta \left[ \textup{e}^{\textup{i}\beta A_s}V\textup{e}^{\textup{i}(t-\beta) A_s} - \textup{e}^{\textup{i}\beta A}V\textup{e}^{\textup{i}(t-\beta)A} \right]\\
& \hspace{-10.8cm} = \int_{-\infty}^{\infty} \textup{i} \hat{\phi}(t) ~Q(t)~ dt,\quad \text{where} \quad Q(t) \equiv \int_0^1ds \int_0^td\beta \left[ \textup{e}^{\textup{i}\beta A_s}V\textup{e}^{\textup{i}(t-\beta) A_s} - \textup{e}^{\textup{i}\beta A}V\textup{e}^{\textup{i}(t-\beta)A} \right].
\end{split}
\end{equation}
As before, $\tau \in[0,1] \longrightarrow \textup{e}^{\textup{i}\beta A_{\tau}}V\textup{e}^{\textup{i}(t-\beta) A_{\tau}}\in \mathcal{B}_3(\mathcal{H})$ is $\mathcal{B}_{\frac{3}{2}}(\mathcal{H})-$
continuously differentiable, uniformly with respect to $\beta$. Then 
\begin{equation*}
\begin{split}
Q(t) = \int_0^1ds \int_0^td\beta \int_0^sd\tau \frac{d}{d\tau}\left( \textup{e}^{\textup{i}\beta A_{\tau}}V\textup{e}^{\textup{i}(t-\beta) A_{\tau}} \right)\\
& \hspace{-8.5cm} = \textup{i} \int_0^1ds \int_0^sd\tau \int_0^td\beta ~\{\int_0^{\beta}d\nu \textup{e}^{\textup{i}\nu A_{\tau}}V\textup{e}^{\textup{i}(\beta-\nu)A_{\tau}}V\textup{e}^{\textup{i}(t-\beta)A_{\tau}}
+ \int_0^{t-\beta}d\nu \textup{e}^{\textup{i}\beta A_{\tau}}V\textup{e}^{\textup{i}\nu A_{\tau}}V\textup{e}^{\textup{i}(t-\beta-\nu)A_{\tau}}\},
\end{split}
\end{equation*}
where Fubini's theorem has been used to interchange the order of integration. Hence
\begin{equation}\label{eq: diffsecnottrace}
\begin{split}
\phi(A+V) - \phi(A) -\left[D^{(1)}\phi(A)\right](V) - \frac{1}{2}\left[D^{(2)}\phi(A)\right](V,V) \\
& \hspace{-10.5cm} = \int_{-\infty}^{\infty} \textup{i}^2 \hat{\phi}(t)dt \int_0^1ds \int_0^sd\tau \int_0^td\beta \{ \int_0^{\beta}d\nu \left[ \textup{e}^{\textup{i}\nu A_{\tau}}V\textup{e}^{\textup{i}(\beta-\nu)A_{\tau}}Ve^{\textup{i}(t-\beta)A_{\tau}} -\textup{e}^{\textup{i}\nu A}V\textup{e}^{\textup{i}(\beta-\nu)A}V\textup{e}^{\textup{i}(t-\beta)A}\right]\\
& \hspace{-5cm}  + \int_0^{t-\beta}d\nu \left[ \textup{e}^{\textup{i}\beta A_{\tau}}V\textup{e}^{\textup{i}\nu A_{\tau}}V\textup{e}^{\textup{i}(t-\beta-\nu)A_{\tau}} -\textup{e}^{\textup{i}\beta A}V\textup{e}^{\textup{i}\nu A}V\textup{e}^{\textup{i}(t-\beta-\nu)A}\right]\}\\
\end{split}
\end{equation}
Though each of the four individual terms in the integral in \eqref{eq: diffsecnottrace} belong to $\mathcal{B}_{\frac{3}{2}}(\mathcal{H})$, each of the 
differences in parenthesis $[.]$ belong to $\mathcal{B}_1(\mathcal{H})$, e.g.
\begin{equation*}
\begin{split}
\left[ \textup{e}^{\textup{i}\nu A_{\tau}}V\textup{e}^{\textup{i}(\beta-\nu)A_{\tau}}V\textup{e}^{\textup{i}(t-\beta)A_{\tau}} -\textup{e}^{\textup{i}\nu A}V\textup{e}^{\textup{i}(\beta-\nu)A}V\textup{e}^{\textup{i}(t-\beta)A}\right]\\
& \hspace{-8cm} = \left[\textup{e}^{\textup{i}\nu A_{\tau}}-\textup{e}^{\textup{i}\nu A}\right]Ve^{\textup{i}(\beta-\nu)A_{\tau}}V\textup{e}^{\textup{i}(t-\beta)A_{\tau}} + \textup{e}^{\textup{i}\nu A}V\left[\textup{e}^{\textup{i}(\beta-\nu)A_{\tau}}-\textup{e}^{\textup{i}(\beta-\nu)A_{\tau}}\right]V\textup{e}^{\textup{i}(t-\beta)A_{\tau}}\\
& \hspace{-4cm} + \textup{e}^{\textup{i}(\beta-\nu)A_{\tau}}V\left[\textup{e}^{\textup{i}(t-\beta)A_{\tau}}-\textup{e}^{\textup{i}(t-\beta)A_{\tau}}\right] \in \mathcal{B}_1(\mathcal{H}),
\end{split}
\end{equation*}
and its norm $\|[.]\|_1 ~\leq ~|\tau|~\|V\|_3^3 ~\{|\nu| + |\beta -\nu| + |t-\beta| \},$ where we have used the estimate: $\left\|\textup{e}^{\textup{i}\nu A_{\tau}}-\textup{e}^{\textup{i}\nu A}\right\|_3 \leq |\nu \tau| \|V\|_3$.
Hence by the hypothesis on $\hat \phi$,
$$ \phi(A+V) - \phi(A) -\left[D^{(1)}\phi(A)\right](V) - \frac{1}{2}\left[D^{(2)}\phi(A)\right](V,V) \in \mathcal{B}_1(\mathcal{H})\quad \text{and} \quad$$
\begin{equation}\label{eq: secondifftrace}
\begin{split}
\mathcal{Z} \equiv \textup{Tr}\left[\phi(A+V) - \phi(A) -\left[D^{(1)}\phi(A)\right](V) - \frac{1}{2}\left[D^{(2)}\phi(A)\right](V,V)\right]  \\
& \hspace{-12cm} = \int_{-\infty}^{\infty} \textup{i}^2 \hat{\phi}(t)dt \int_0^1ds \int_0^sd\tau \int_0^td\beta ~\{ \int_0^{\beta}d\nu \textup{Tr} \left[ \textup{e}^{\textup{i}\nu A_{\tau}}V\textup{e}^{\textup{i}(\beta-\nu)A_{\tau}}V\textup{e}^{\textup{i}(t-\beta)A_{\tau}} -\textup{e}^{\textup{i}\nu A}V\textup{e}^{\textup{i}(\beta-\nu)A}V\textup{e}^{\textup{i}(t-\beta)A}\right]\\
& \hspace{-6.5cm}  + \int_0^{t-\beta}d\nu \textup{Tr} \left[ \textup{e}^{\textup{i}\beta A_{\tau}}V\textup{e}^{\textup{i}\nu A_{\tau}}V\textup{e}^{\textup{i}(t-\beta-\nu)A_{\tau}} -\textup{e}^{\textup{i}\beta A}V\textup{e}^{\textup{i}\nu A}V\textup{e}^{\textup{i}(t-\beta-\nu)A}\right]\}.
\end{split}
\end{equation}
Again by the cyclicity of trace and a change of variable, the first integral in $\{.\}$ in \eqref{eq: secondifftrace} is equal to 
\begin{equation}\label{eq: 1stintegral}
\begin{split}
\int_0^{\beta}d\nu~ \textup{Tr} \left[ \textup{e}^{\textup{i}(t-\beta +\nu) A_{\tau}}V\textup{e}^{\textup{i}(\beta-\nu)A_{\tau}}V -\textup{e}^{\textup{i}(t-\beta +\nu) A}V\textup{e}^{\textup{i}(\beta-\nu)A}V\right]\\
& \hspace{-6cm} = \int_0^{\beta}d\nu ~\textup{Tr} \left[ V\textup{e}^{\textup{i}(t- \nu) A_{\tau}}V\textup{e}^{\textup{i}\nu A_{\tau}} -V\textup{e}^{\textup{i}(t- \nu) A}V\textup{e}^{\textup{i}\nu A}\right].
\end{split}
\end{equation}
Similarly, the second integral in $\{.\}$ in \eqref{eq: secondifftrace} is equal to 
\begin{equation}\label{eq: secondintegral}
\hspace{0cm}  \int_0^{t-\beta}d\nu ~\textup{Tr} \left[ V\textup{e}^{\textup{i}(t- \nu) A_{\tau}}V\textup{e}^{\textup{i}\nu A_{\tau}} -V\textup{e}^{\textup{i}(t- \nu) A}V\textup{e}^{\textup{i}\nu A}\right].
\end{equation}
Combining \eqref{eq: 1stintegral} and \eqref{eq: secondintegral}, we conclude that 
\begin{equation}\label{eq: mathcalz}
\begin{split}
\mathcal{Z} = \int_{-\infty}^{\infty} \textup{i}^2 \hat{\phi}(t)dt \int_0^1ds \int_0^sd\tau \int_0^td\beta ~\{ \int_0^{\beta}d\nu ~\textup{Tr} \left[ V\textup{e}^{\textup{i}(t- \nu) A_{\tau}}V\textup{e}^{\textup{i}\nu A_{\tau}} -V\textup{e}^{\textup{i}(t- \nu) A}V\textup{e}^{\textup{i}\nu A}\right]\\
& \hspace{-8cm}  + \int_0^{t-\beta}d\nu ~\textup{Tr} \left[ V\textup{e}^{\textup{i}(t- \nu) A_{\tau}}V\textup{e}^{\textup{i}\nu A_{\tau}} -V\textup{e}^{\textup{i}(t- \nu) A}V\textup{e}^{\textup{i}\nu A}\right]\}.
\end{split}
\end{equation}
By a change of variable and the cyclicity of trace, we get that
\begin{equation*}
\begin{split} 
t \int_0^td\nu ~\textup{Tr} \left[ V\textup{e}^{\textup{i}(t- \nu) A_{\tau}}V\textup{e}^{\textup{i}\nu A_{\tau}} -V\textup{e}^{\textup{i}(t- \nu) A}V\textup{e}^{\textup{i}\nu A}\right]\\
& \hspace{-6cm} = \int_0^t d\beta ~\{ \int_0^{\beta}d\nu ~\textup{Tr} \left[ V\textup{e}^{\textup{i}(t- \nu) A_{\tau}}V\textup{e}^{\textup{i}\nu A_{\tau}} -V\textup{e}^{\textup{i}(t- \nu) A}V\textup{e}^{\textup{i}\nu A}\right]\\
& \hspace{-3cm} +  \int_0^{t-\beta} d\nu ~\textup{Tr} \left[ V\textup{e}^{\textup{i}(t- \nu) A_{\tau}}V\textup{e}^{\textup{i}\nu A_{\tau}} -V\textup{e}^{\textup{i}(t- \nu) A}V\textup{e}^{\textup{i}\nu A}\right]\},
\end{split}
\end{equation*}
using which in \eqref{eq: mathcalz} we are lead to the equation 
\begin{equation}\label{eq: schfrechexpr}
\mathcal{Z} = \int_{-\infty}^{\infty} \textup{i}^2 ~t \hat{\phi}(t)dt\int_0^td\nu \int_0^1ds \int_0^sd\tau  ~\textup{Tr} \left[ V\textup{e}^{\textup{i}(t- \nu) A_{\tau}}V\textup{e}^{\textup{i}\nu A_{\tau}} -V\textup{e}^{\textup{i}(t- \nu) A}V\textup{e}^{\textup{i}\nu A}\right],
\end{equation}
by an application of Fubini's theorem.
~~~~~~~~~~~~~~~~~~~~~~~~~~~~~~~~~~~~~~~~~~~~~~~~~~~~~~~~~~~~~~~~~~~~~~~~~~~~~~~~~~~~~~~~~~~~~~\end{proof}
\begin{thm}\label{thm: mainunbthm}
Let $A$ be an unbounded self-adjoint operator in a Hilbert space $\mathcal{H}$ with $\sigma(A) \subseteq [b,\infty)$ for some $b\in \mathbb{R}$ and $V$ be a self-adjoint operator such that $V\in \mathcal{B}_2(\mathcal{H})$. 
Then there exist a unique real-valued function $\eta\in L^1\left(\mathbb{R},\frac{d\lambda}{(1+\lambda^2)^{1+\epsilon}}\right)$ (for some $\epsilon > 0$) such that for every $\phi \in \mathcal{S}(\mathbb{R})$ (the Schwartz class of smooth functions of rapid decrease)
$$\textup{Tr}\left[\phi(A+V) - \phi(A) -\left[D^{(1)}\phi(A)\right](V) - \frac{1}{2}\left[D^{(2)}\phi(A)\right](V,V)\right] = \int_{-\infty}^{\infty} \phi'''(\lambda)\eta(\lambda)d\lambda .$$
\end{thm}
\begin{proof}
Equation \eqref{eq: schfrechexpr}, after an application of Fubini's theprem, yields that
\begin{equation}\label{eq : totalexp}
\begin{split} 
\mathcal{Z} \equiv \textup{Tr}\left[\phi(A+V) - \phi(A) -\left[D^{(1)}\phi(A)\right](V) - \frac{1}{2}\left[D^{(2)}\phi(A)\right](V,V)\right]\\
& \hspace{-10.2cm} = \int_0^1ds \int_0^sd\tau \int_{-\infty}^{\infty} \textup{i}^2 ~t \hat{\phi}(t)dt\int_0^td\nu ~\textup{Tr} \left[ V\textup{e}^{\textup{i}(t- \nu) A_{\tau}}V\textup{e}^{\textup{i}\nu A_{\tau}} -V\textup{e}^{\textup{i}(t- \nu) A}V\textup{e}^{\textup{i}\nu A}\right].
\end{split}
\end{equation}
Now
\begin{equation}\label{eq: specthmeq}
\begin{split} 
\int_0^td\nu ~\textup{Tr} \left[ V\textup{e}^{\textup{i}(t- \nu) A_{\tau}}V\textup{e}^{\textup{i}\nu A_{\tau}} -V\textup{e}^{\textup{i}(t- \nu) A}V\textup{e}^{\textup{i}\nu A}\right]\\
& \hspace{-5cm} =  \int_0^td\nu \int_{a}^{\infty} \int_{a}^{\infty} \textup{e}^{\textup{i}(t-\nu)\lambda}\textup{e}^{\textup{i}\nu \mu}~\textup{Tr} \left[VE_{\tau}(d\lambda)VE_{\tau}(d\mu) -VE(d\lambda)VE(d\mu)\right],
\end{split}
\end{equation}
where $a =\min\{ b, ~\inf{\sigma(A_{\tau})} ~[0 < \tau\leq 1]\}$ and $E_{\tau}(.)$ and $E(.)$ are the spectral families of the operator $A_{\tau}$ and $A$ respectively and the measure $ \mathcal{G}: \Delta \times \delta \subseteq Borel(\mathbb{R}^2)  \longrightarrow \textup{Tr} \left[VE_{\tau}(\Delta)VE_{\tau}(\delta) -VE(\Delta)VE(\delta)\right]$
is a complex measure with total variation  $\leq  2\|V\|_2^2$ and hence by Fubini's theorem  the right hand side expression in  \eqref{eq: specthmeq} is 
equal to  
\begin{equation}\label{eq: trv1v2eq}
\begin{split} 
 \int_{a}^{\infty}\int_{a}^{\infty} ~\int_0^td\nu ~\textup{e}^{\textup{i}(t-\nu)\lambda}\textup{e}^{\textup{i}\nu \mu}~\textup{Tr} \left[VE_{\tau}(d\lambda)VE_{\tau}(d\mu) -VE(d\lambda)VE(d\mu)\right]\\
& \hspace{-12.5cm} =  \int_{a}^{\infty}\int_{a}^{\infty} ~\frac{\textup{e}^{\textup{i} t\lambda} - \textup{e}^{\textup{i} t\mu}}{\textup{i}(\lambda-\mu)} ~~\textup{Tr} \left[VE_{\tau}(d\lambda)VE_{\tau}(d\mu) -VE(d\lambda)VE(d\mu)\right]\\
& \hspace{-12.5cm} = \int_0^1ds \int_0^sd\tau \int_{-\infty}^{\infty} \textup{i}^2 ~t \hat{\phi}(t)dt \int_{a}^{\infty} t\textup{e}^{\textup{i}t\lambda} ~\textup{Tr}\left[ V_{1\tau}^2E_{\tau}(d\lambda) - V_1^2E(d\lambda)\right]\\
& \hspace{-11.8cm} + \int_0^1ds \int_0^sd\tau \int_{-\infty}^{\infty} \textup{i}^2 ~t \hat{\phi}(t)dt \int_{a}^{\infty}\int_{a}^{\infty} \frac{\textup{e}^{\textup{i} t\lambda} - \textup{e}^{\textup{i} t\mu}}{\textup{i}(\lambda-\mu)} ~\textup{Tr} \left[V_{2\tau}E_{\tau}(d\lambda)V_{2\tau}E_{\tau}(d\mu) -V_2E(d\lambda)V_2E(d\mu)\right],
\end{split}
\end{equation}
where we have set $V = V_1 \oplus V_2 = V_{1\tau} \oplus V_{2\tau} \in \mathcal{B}_2(\mathcal{H})$  as in Lemma \ref{lmma: decomp} $vi(b)$.
Applying Fubini's theorem in the first expression in right hand side of  \eqref{eq: trv1v2eq}, we conclude that
\begin{equation*}
\begin{split} 
\int_0^1ds \int_0^sd\tau  \int_{-\infty}^{\infty}~\textup{i}^2 ~t \hat{\phi}(t) dt \int_{a}^{\infty} ~t\textup{e}^{\textup{i}t\lambda}\textup{Tr}\left[ V_{1\tau}^2E_{\tau}(d\lambda) - V_1^2E(d\lambda)\right]\\  
& \hspace{-8cm} = \int_0^1ds \int_0^sd\tau \int_{a}^{\infty} \phi''(\lambda)~\textup{Tr}\left[ V_{1\tau}^2E_{\tau}(d\lambda)-V_1^2E(d\lambda)\right]\\
& \hspace{-8cm} = \int_0^1ds \int_0^sd\tau \int_{a}^{\infty} \phi'''(\lambda)~\textup{Tr}\left[ V_1^2E(\lambda)-V_{1\tau}^2E_{\tau}(\lambda)\right]d\lambda,
\end{split}
\end{equation*}
where we have integrated by-parts and observed that the boundary term vanishes. Thus the first term in \eqref{eq: trv1v2eq} is  equal to 
\begin{equation}\label{eq: firsteta1eq}
\int_0^1ds \int_0^s d\tau \int_{a}^{\infty} \phi'''(\lambda) ~\eta_{1\tau}(\lambda) ~d\lambda, \quad \text{where} \quad \eta_{1\tau}(\lambda) = \textup{Tr}\left[ V_1^2E(\lambda)-V_{1\tau}^2E_{\tau}(\lambda)\right].
\end{equation}
Now consider the second expression in the right hand side of \eqref{eq: trv1v2eq} :
\begin{equation}\label{eq: secmaineq}
\int_0^1ds \int_0^sd\tau \int_{-\infty}^{\infty}\textup{i}^2 ~t \hat{\phi}(t)dt \int_{a}^{\infty}\int_{a}^{\infty} \frac{\textup{e}^{\textup{i} t\lambda} - \textup{e}^{\textup{i} t\mu}}{\textup{i}(\lambda-\mu)} ~\textup{Tr} \left[V_{2\tau}E_{\tau}(d\lambda)V_{2\tau}E_{\tau}(d\mu) -V_2E(d\lambda)V_2E(d\mu)\right]
\end{equation}
Since $V_{2\tau}\in \left[\textup{Ker} \left(\mathcal{M}_{\left(A_{\tau}+\textup{i}\right)^{-1}}\right)\right]^{\perp} = \overline{\textup{Ran} \left(\mathcal{M}_{\left(A_{\tau}+\textup{i}\right)^{-1}}^*\right)} = \overline{\textup{Ran} \left(\mathcal{M}_{\left(A_{\tau} - \textup{i}\right)^{-1}}\right)}$ , there exists a sequence
$\{V_{2\tau}^{(n)}\} \subseteq \textup{Ran} \left(\mathcal{M}_{\left(A_{\tau} - \textup{i}\right)^{-1}}\right)$ i.e. $V_{2\tau}^{(n)} = \left(A_{\tau} - \textup{i}\right)^{-1} Y^{(n)} - Y^{(n)} \left(A_{\tau} - \textup{i}\right)^{-1}$ 
for some $Y^{(n)} \in \mathcal{B}_2(\mathcal{H})$  such that $V_{2\tau}^{(n)} \longrightarrow V_{2\tau}$ in $\|.\|_2$. Similarly, there exists a sequence
$ \{V_{2}^{(n)}\} \subseteq \textup{Ran} \left(\mathcal{M}_{\left(A - \textup{i}\right)^{-1}}\right)$ i.e. $V_{2}^{(n)} = \left(A - \textup{i}\right)^{-1} Y_0^{(n)} - Y_0^{(n)} \left(A - \textup{i}\right)^{-1}$ 
for $Y_0^{(n)} \in \mathcal{B}_2(\mathcal{H})$  such that $V_{2}^{(n)} \longrightarrow V_{2}$ in $\|.\|_2$. Furthermore, by lemma \ref{lmma: decomp} $(vii)(b),$
the map $[0,1]\ni \tau \longrightarrow V_{1\tau},~V_{2\tau}$ are continuous. Thus the expression in \eqref{eq: secmaineq} is equal to 
\begin{equation}\label{eq: thirdmaineq}
\begin{split} 
\int_0^1ds \int_0^sd\tau \int_{-\infty}^{\infty} \textup{i}^2 ~t \hat{\phi}(t)dt \int_{a}^{\infty}\int_{a}^{\infty} \frac{\textup{e}^{\textup{i} t\lambda} - \textup{e}^{\textup{i} t\mu}}{\textup{i}(\lambda-\mu)} \lim_{n\rightarrow \infty}~\textup{Tr} \left[V_{2\tau}E_{\tau}(d\lambda)V_{2\tau}^{(n)}E_{\tau}(d\mu) -V_2E(d\lambda)V_2^{(n)}E(d\mu)\right]\\
& \hspace{-18.3cm} = \int_0^1ds \int_0^sd\tau \int_{-\infty}^{\infty} \textup{i}^2 ~t \hat{\phi}(t)dt \lim_{n\rightarrow \infty} \int_{a}^{\infty}\int_{a}^{\infty} \frac{\textup{e}^{\textup{i} t\lambda} - \textup{e}^{\textup{i} t\mu}}{\textup{i}(\lambda-\mu)} ~\textup{Tr} \left[V_{2\tau}E_{\tau}(d\lambda)V_{2\tau}^{(n)}E_{\tau}(d\mu) -V_2E(d\lambda)V_2^{(n)}E(d\mu)\right],
\end{split}
\end{equation}
since $Var\left(\mathcal{G}_2^{(n)} -\mathcal{G}_2\right) \leq ~\|V_{2\tau}\|_2 ~\|V_{2\tau}^{(n)} - V_{2\tau}\|_2 + \|V\|_2 ~\|V_2 - V_2^{(n)}\| \longrightarrow 0$
as $n\longrightarrow \infty$, where $\mathcal{G}_2 (\Delta \times \delta)= \textup{Tr} \left[V_{2\tau}E_{\tau}(\Delta)V_{2\tau}E_{\tau}(\delta) -V_2E(\Delta)V_2E(\delta)\right]$ and
$\mathcal{G}_2^{(n)}\left(\Delta \times \delta\right)$ is the same expression with second $V_2$-terms replaced by $V_2^{(n)}$. These  are complex measures on $\mathbb{R}^2$
and $Var\left(\mathcal{G}_2^{(n)} -\mathcal{G}_2\right)$ is the variation of $\left(\mathcal{G}_2^{(n)} -\mathcal{G}_2\right)$. Note that
\begin{equation*}
\begin{split} 
\hspace{-2cm} \textup{Tr} \left(V_{2\tau}E_{\tau}(d\lambda)V_{2\tau}^{(n)}E_{\tau}(d\mu)\right) = \textup{Tr} \left(V_{2\tau}E_{\tau}(d\lambda)\left[ \left(A_{\tau} - \textup{i}\right)^{-1} Y^{(n)} - Y^{(n)} \left(A_{\tau} - \textup{i}\right)^{-1} \right]E_{\tau}(d\mu)\right]\\
& \hspace{-10.8cm} =  ~\frac{-(\lambda -\mu)}{(\lambda -\textup{i})(\mu - \textup{i})} ~\textup{Tr} \left(V_{2\tau}E_{\tau}(d\lambda)Y^{(n)}E_{\tau}(d\mu)\right)
\end{split}
\end{equation*}
and since $\int_{-\infty}^{\infty} |t\hat{\phi}(t)|dt < \infty$ and  the other functions are bounded, the right hand side expression in 
\eqref{eq: thirdmaineq} is equal to 
\begin{equation}\label{eq: fourthmaineq}
\begin{split} 
\int_0^1ds \int_0^sd\tau \lim_{n\rightarrow \infty} \int_{-\infty}^{\infty} \textup{i}^2 ~t \hat{\phi}(t)dt  \int_{a}^{\infty}\int_{a}^{\infty} \frac{\textup{e}^{\textup{i} t\lambda} - \textup{e}^{\textup{i} t\mu}}{\textup{i}(\lambda-\mu)} \left[\frac{-(\lambda -\mu)}{(\lambda -\textup{i})(\mu - \textup{i})}\right]~\textup{Tr} \{V_{2\tau}E_{\tau}(d\lambda)Y^{(n)}E_{\tau}(d\mu) \\
& \hspace{-3.5cm} -V_2E(d\lambda)Y_0^{(n)}E(d\mu)\}\\
& \hspace{-17cm} = \int_0^1ds \int_0^sd\tau \lim_{n\rightarrow \infty} \int_{-\infty}^{\infty}-\textup{i} t \hat{\phi}(t)dt  \int_{a}^{\infty}\int_{a}^{\infty} \left[\textup{e}^{\textup{i} t\lambda} - \textup{e}^{\textup{i} t\mu}\right] \textup{Tr} \{V_{2\tau}E_{\tau}(d\lambda)\left(A_{\tau}-\textup{i}\right)^{-1}Y^{(n)}\left(A_{\tau}-\textup{i}\right)^{-1}E_{\tau}(d\mu)\\
& \hspace{-6.5cm} -V_2E(d\lambda)\left(A-\textup{i}\right)^{-1}Y_0^{(n)}\left(A-\textup{i}\right)^{-1}E(d\mu)\}\\
& \hspace{-17cm} = \int_0^1ds \int_0^sd\tau \lim_{n\rightarrow \infty} \int_{-\infty}^{\infty}-\textup{i} t \hat{\phi}(t)dt  \int_{a}^{\infty} ~\textup{e}^{\textup{i} t\lambda} ~\textup{Tr} \left(V_{2\tau}\left[E_{\tau}(d\lambda),\tilde{Y}^{(n)}\right] - V_2\left[E(d\lambda),\tilde{Y_0}^{(n)}\right]\right),
\end{split}
\end{equation}
where $\tilde{Y}^{(n)} = \left(A_{\tau}-\textup{i}\right)^{-1}Y^{(n)}\left(A_{\tau}-\textup{i}\right)^{-1}$ and $\tilde{Y_0}^{(n)} = \left(A-\textup{i}\right)^{-1}Y_0^{(n)}\left(A-\textup{i}\right)^{-1}$.
Again by applying Fubini's theorem the right hand side expression in \eqref{eq: fourthmaineq} is equal to 
$$ \int_0^1ds \int_0^sd\tau \lim_{n\rightarrow \infty} ~\int_{a}^{\infty} -\phi'(\lambda) ~\textup{Tr} \left(V_{2\tau}\left[E_{\tau}(d\lambda),\tilde{Y}^{(n)}\right] - V_2\left[E(d\lambda),\tilde{Y_0}^{(n)}\right]\right),$$
and by integrating by-parts twice and on observing that the boundary term vanishes, this  reduces to
\begin{equation*}
\begin{split} 
\int_0^1ds \int_0^sd\tau \lim_{n\rightarrow \infty} ~ -\{\phi'(\lambda) ~\textup{Tr} (V_{2\tau}[E_{\tau}(\lambda),\tilde{Y}^{(n)}] - V_2 [E(\lambda),\tilde{Y_0}^{(n)}])\mid_{\lambda =a}^{\infty}\\
& \hspace{-8cm} - \int_{a}^{\infty} -\phi''(\lambda) ~\textup{Tr} \left(V_{2\tau}\left[E_{\tau}(\lambda),\tilde{Y}^{(n)}\right] - V_2\left[E(\lambda),\tilde{Y_0}^{(n)}\right]\right) d\lambda\}\\
& \hspace{-13cm} = \int_0^1ds \int_0^sd\tau \lim_{n\rightarrow \infty} ~\int_{a}^{\infty} \phi''(\lambda) ~\textup{Tr} \left(V_{2\tau}\left[E_{\tau}(\lambda),\tilde{Y}^{(n)}\right] - V_2\left[E(\lambda),\tilde{Y_0}^{(n)}\right]\right) d\lambda\\
& \hspace{-13cm} = \int_0^1ds \int_0^sd\tau \lim_{n\rightarrow \infty} ~\{\phi''(\lambda) \int_a^{\lambda} \textup{Tr} (V_{2\tau}[E_{\tau}(\mu),\tilde{Y}^{(n)}] - V_2[E(\mu),\tilde{Y_0}^{(n)}]) ~d\mu \mid_{\lambda = a}^{\infty}\\
& \hspace{-9cm} - \int_{a}^{\infty} \phi'''(\lambda) \left( \int_a^{\lambda} ~\textup{Tr} \left(V_{2\tau}\left[E_{\tau}(\mu),\tilde{Y}^{(n)}\right] - V_2\left[E(\mu),\tilde{Y_0}^{(n)}\right]\right) d\mu \right) d\lambda\}\\
& \hspace{-13cm} = \int_0^1ds \int_0^sd\tau \lim_{n\rightarrow \infty} ~\int_{a}^{\infty} \phi'''(\lambda) \left( \int_a^{\lambda} ~\textup{Tr} \left(V_{2}\left[E(\mu),\tilde{Y_0}^{(n)}\right] - V_{2\tau}\left[E_{\tau}(\mu),\tilde{Y}^{(n)}\right]\right) d\mu \right) d\lambda\\
\end{split}
\end{equation*}
\begin{equation}\label{eq: seceta2eq}
\hspace{-9cm} = \int_0^1ds \int_0^sd\tau \lim_{n\rightarrow \infty} ~\int_{a}^{\infty} \phi'''(\lambda) ~\eta_{2\tau}^{(n)}(\lambda)~ d\lambda, 
 \end{equation}
$$\hspace{-5cm} \quad \text{where} \quad \eta_{2\tau}^{(n)}(\lambda) =  \int_a^{\lambda} ~\textup{Tr} \left(V_{2}\left[E(\mu),\tilde{Y_0}^{(n)}\right] - V_{2\tau}\left[E_{\tau}(\mu),\tilde{Y}^{(n)}\right]\right) d\mu. $$
Here it is worth observing that the hypothesis that $A$ is bounded below is used for the first time here, only for performing the second integration-by-parts.
Combining \eqref{eq: firsteta1eq} and \eqref{eq: seceta2eq}, we conclude that 
\begin{equation}\label{eq: bfrlimit}
\begin{split}
\textup{Tr}\left[\phi(A+V) - \phi(A) -\left[D^{(1)}\phi(A)\right](V) - \frac{1}{2}\left[D^{(2)}\phi(A)\right](V,V)\right] \\
& \hspace{-7cm} = \int_0^1ds \int_0^sd\tau \lim_{n\rightarrow \infty} ~\int_{a}^{\infty} \phi'''(\lambda) ~\eta_{\tau}^{(n)}(\lambda)~ d\lambda,
\end{split}
\end{equation}
where ~$\eta_{\tau}^{(n)}(\lambda) = \eta_{1\tau}(\lambda) + \eta_{2\tau}^{(n)}(\lambda)$. We claim that $\{\eta_{\tau}^{(n)}\}$
is a cauchy sequence in 

\hspace{-0.8cm} $ L^1\left(\mathbb{R},\frac{d\lambda}{(1+\lambda^2)^{1+\epsilon}}\right)$~ ($\epsilon > 0$) and we follow the idea from \cite{gesztesykop}. First note that  

\hspace{-0.8cm} $L^{\infty}(\mathbb{R},d\lambda) = L^{\infty}(\mathbb{R},\psi(\lambda)d\lambda)$ [~where $\psi(\lambda) = \frac{1}{(1+\lambda^2)^{1+\epsilon}}$~] since the two measures are 
equivalent. Next, let $f\in L^{\infty}(\mathbb{R},d\lambda)$ and define
$$\hspace{-0.5cm}  g(\lambda) = \int_{\lambda}^{\infty} f(t) \psi(t) dt \quad \text{for} \quad \lambda \in \mathbb{R}, \quad \text{then} \quad g \quad \text{is absolutely continuous with} \quad $$
$g'(\lambda) = -f(\lambda)\psi(\lambda)$ a.e. and that $|g(\lambda)|\leq \quad \text{Const.} \quad \frac{1}{(1+\lambda^2)^{\frac{1}{2}+\epsilon'}}$ ~~(for some $\epsilon' >0)$ for $\lambda \rightarrow \infty$ and bounded.
Next consider the expression 
\begin{equation*}
\begin{split} 
\int_{-\infty}^{\infty} f(\lambda)\psi(\lambda)\left[\eta_{\tau}^{(n)}(\lambda) - \eta_{\tau}^{(m)}(\lambda) \right] d\lambda = \int_{a}^{\infty} f(\lambda)\psi(\lambda)\left[\eta_{2\tau}^{(n)}(\lambda) - \eta_{2\tau}^{(m)}(\lambda) \right] d\lambda\\
& \hspace{-13cm} = \int_{a}^{\infty} -g'(\lambda)~ d\lambda \left(\int_a^{\lambda} ~\textup{Tr} \left(V_{2}\left[E(\mu),\tilde{Y_0}^{(n)} -\tilde{Y_0}^{(m)}\right] - V_{2\tau}\left[E_{\tau}(\mu),\tilde{Y}^{(n)} -\tilde{Y}^{(m)}\right]\right) d\mu \right),
\end{split}
\end{equation*}
which on integration by-parts and on observing that the boundary terms vanishes, leads to 
\begin{equation}\label{eq: intbyparteq}
\begin{split} 
\int_{a}^{\infty} g(\lambda) \textup{Tr} \left(V_{2}\left[E(\lambda),\tilde{Y_0}^{(n)} -\tilde{Y_0}^{(m)}\right] - V_{2\tau}\left[E_{\tau}(\lambda),\tilde{Y}^{(n)} -\tilde{Y}^{(m)}\right]\right)d\lambda.
\end{split}
\end{equation}
Define
$$ h(\lambda) = \int_a^{\lambda} g(t)dt \quad \text{for} \quad \lambda \in [a,\infty), \quad \text{then} \quad h \quad \text{is bounded, differentiable } \quad [a,\infty) $$
with $h'(\lambda) = g(\lambda) ~\forall ~\lambda \in [a,\infty)$. 
Hence by integrating by-parts and  observing that the boundary term vanishes, the right hand side expression in \eqref{eq: intbyparteq} is equal to
\begin{equation}\label{eq: combunbb1}
\begin{split} 
\int_{a}^{\infty} h'(\lambda) \textup{Tr} \left(V_{2}\left[E(\lambda),\tilde{Y_0}^{(n)} -\tilde{Y_0}^{(m)}\right] - V_{2\tau}\left[E_{\tau}(\lambda),\tilde{Y}^{(n)} -\tilde{Y}^{(m)}\right]\right)d\lambda\\
& \hspace{-11.5cm} = ~h(\lambda) ~\textup{Tr} \left(V_{2}\left[E(\lambda),\tilde{Y_0}^{(n)} -\tilde{Y_0}^{(m)}\right] - V_{2\tau}\left[E_{\tau}(\lambda),\tilde{Y}^{(n)} -\tilde{Y}^{(m)}\right]\right) \mid_{\lambda = a}^{\infty}\\
& \hspace{-9cm} - \int_{a}^{\infty} h(\lambda) \textup{Tr} \left(V_{2}\left[E(d\lambda),\tilde{Y_0}^{(n)} -\tilde{Y_0}^{(m)}\right] - V_{2\tau}\left[E_{\tau}(d\lambda),\tilde{Y}^{(n)} -\tilde{Y}^{(m)}\right]\right)\\
& \hspace{-11.5cm} = \int_{a}^{\infty} h(\lambda) ~\textup{Tr} \left(V_{2\tau}\left[E_{\tau}(d\lambda),\tilde{Y}^{(n)} -\tilde{Y}^{(m)}\right] - V_{2}\left[E(d\lambda),\tilde{Y_0}^{(n)} -\tilde{Y_0}^{(m)}\right]\right)\\
& \hspace{-11.5cm} = ~\textup{Tr} \left(V_{2\tau}\left[h(A_{\tau}),\tilde{Y}^{(n)} -\tilde{Y}^{(m)}\right] - V_{2}\left[h(A),\tilde{Y_0}^{(n)} -\tilde{Y_0}^{(m)}\right]\right).
\end{split}
\end{equation}
But on the other hand,
\begin{equation*}
\begin{split}
\left[h(A),\tilde{Y_0}^{(n)}\right] = \int_a^{\infty}\int_a^{\infty}[h(\lambda)-h(\mu)]E(d\lambda)\tilde{Y_0}^{(n)}E(d\mu) \\
&\hspace{-7.5cm}   = \int_a^{\infty}\int_a^{\infty}[h(\lambda)-h(\mu)](\lambda-\textup{i})^{-1} (\mu -\textup{i})^{-1}E(d\lambda)Y_0^{(n)}E(d\mu)\\ 
& \hspace{-7.5cm}  = \int_a^{\infty}\int_a^{\infty}\frac{h(\lambda)-h(\mu)}{(\lambda-\textup{i})^{-1} -(\mu-\textup{i})^{-1}}(\lambda-\textup{i})^{-1} (\mu -\textup{i})^{-1} E(d\lambda)V_2^{(n)}E(d\mu)
\end{split}
\end{equation*}
and hence
$$\left[h(A),\tilde{Y_0}^{(n)}\right] =  - \int_a^{\infty}\int_a^{\infty}\frac{h(\lambda)-h(\mu)}{\lambda -\mu} ~E(d\lambda)V_2^{(n)}E(d\mu). $$
Therefore
\begin{equation}\label{eq: combunbb2} 
\textup{Tr}\left(V_2\left[h(A),\tilde{Y_0}^{(n)}-\tilde{Y_0}^{(n)}\right] \right) = - \int_a^{\infty}\int_a^{\infty}\frac{h(\lambda)-h(\mu)}{\lambda -\mu} ~\textup{Tr}\left(V_2E(d\lambda)\left[V_2^{(n)}-V_2^{(m)}\right]E(d\mu)\right)
\end{equation}
and hence as in Birman-Solomyak (\cite{birmansolomyak},\cite{birsolomyak})  and  in \cite{chttosinha} ,
$$\left| \textup{Tr}\left(V_2\left[h(A),\tilde{Y_0}^{(n)}-\tilde{Y_0}^{(n)}\right] \right) \right| \leq \|f\|_{\infty} ~\|\psi\|_{L^1} ~\|V\|_2 \left\| \left[V_2^{(n)}-V_2^{(m)}\right] \right\|_2$$
and hence 
$$\frac{\left|\int_{-\infty}^{\infty} f(\lambda)\psi(\lambda)\left[\eta_{\tau}^{(n)}(\lambda) - \eta_{\tau}^{(m)}(\lambda) \right] d\lambda \right|}{\|f\|_{\infty}} \leq \|\psi\|_{L^1}~\|V\|_2~\left( \left\| \left[V_2^{(n)}-V_2^{(m)}\right] \right\|_2 + \left\| \left[V_{2\tau}^{(n)}-V_{2\tau}^{(m)}\right] \right\|_2\right)$$
i.e. $$ \| \eta_{\tau}^{(n)} - \eta_{\tau}^{(m)}\|_{L^1(\mathbb{R},\psi(\lambda)d\lambda)} \leq \|\psi\|_{L^1}~\|V\|_2~\left( \left\| \left[V_2^{(n)}-V_2^{(m)}\right] \right\|_2 + \left\| \left[V_{2\tau}^{(n)}-V_{2\tau}^{(m)}\right] \right\|_2\right),$$
which converges to $0$ as $m,n \longrightarrow \infty$ and $\forall ~\tau \in [0,1]$. A similar computation also shows that $\| \eta_{\tau}^{(n)}\|_{L^1(\mathbb{R},\psi(\lambda)d\lambda)} \leq 2~\|\psi\|_{L^1} ~\|V\|_2^2,$
independent of $\tau$ and $n$. Therefore $\{\int_0^1ds\int_0^sd\tau ~\eta_{\tau}^{(n)}(\lambda) \equiv \eta^{(n)}(\lambda)\}$ is also cauchy in $L^1(\mathbb{R},\psi(\lambda)d\lambda)$
and thus there exists a function $\eta \in L^1(\mathbb{R},\psi(\lambda)d\lambda)$ such that $\|\eta^{(n)} - \eta\|_{L^1(\mathbb{R},\psi(\lambda)d\lambda)} \longrightarrow 0$ as $n\longrightarrow \infty,$
by the bounded convergence theorem and hence also $\|\eta\|_{L^1(\mathbb{R},\psi(\lambda)d\lambda)} \leq \|\psi\|_{L^1} \|V\|_2^2.$ Therefore by using the Dominated Convergence theorem as well as Fubini's theorem, from \eqref{eq: bfrlimit}
we conclude that
$$\textup{Tr}\left[\phi(A+V) - \phi(A) -\left[D^{(1)}\phi(A)\right](V) - \frac{1}{2}\left[D^{(2)}\phi(A)\right](V,V)\right] = \int_{-\infty}^{\infty} \phi'''(\lambda)\eta(\lambda)d\lambda.$$
~~~~~~~~~~~~~~~~~~~~~~~~~~~~~~~~~~~~~~~~~~~~~~~~~~~~~~~~~~~~~~~~~~~~~~~~~~~~~~~~~~~~~~~~~~~~~~~~~~~~~~~~~~~~~~~~~~~~~~~~~~~~~~~~~~~~~\end{proof}
The proof of the uniqueness and the real-valued nature of $\eta$ is postponed till after the corollary \ref{cor: 1}.
\begin{crlre}\label{cor: 1}
Let $A$ be an unbounded self-adjoint operator in a Hilbert space $\mathcal{H}$ with $\sigma(A) \subseteq [b,\infty)$ for some $b\in \mathbb{R}$ and $V$ be a self-adjoint operator such that $V\in \mathcal{B}_2(\mathcal{H})$. 
Then the function $\eta \in L^1(\mathbb{R},\psi(\lambda)d\lambda)$ obtained as in theorem \ref{thm: mainunbthm} satisfies the following equation
\begin{equation}\label{eq: cor1eq}
\begin{split}
\int_{-\infty}^{\infty} f(\lambda) \psi(\lambda) \eta(\lambda)d\lambda \\
& \hspace{-4cm} = \int_0^1ds\int_0^sd\tau \int_a^{\infty}\int_a^{\infty}\frac{h(\lambda)-h(\mu)}{\lambda -\mu} ~\textup{Tr}\left[VE(d\lambda)VE(d\mu) - VE_{\tau}(d\lambda)VE_{\tau}(d\mu)\right],
\end{split}
\end{equation}
where $f(\lambda),~ g(\lambda), ~h(\lambda)$ ~and $\psi(\lambda)$ are as in the proof of the theorem \ref{thm: mainunbthm}.
\end{crlre}

\begin{proof}
By Fubini's theorem we have that 
\begin{equation*}
\begin{split} 
\int_{-\infty}^{\infty} f(\lambda)\psi(\lambda) \eta ^{(n)}(\lambda) d\lambda = \int_0^1ds \int_0^sd\tau \int_{a}^{\infty} f(\lambda)\psi(\lambda)  \left[ \eta_{1\tau} (\lambda) + \eta_{2\tau}^{(n)}(\lambda)\right]  d\lambda.
\end{split}
\end{equation*}
But 
\begin{equation*}
\begin{split} 
\int_{a}^{\infty} f(\lambda)\psi(\lambda) \eta_{1 \tau} (\lambda) d\lambda = \int_{a}^{\infty} -g'(\lambda)  \textup{Tr}\left[ V_1^2E(\lambda)-V_{1\tau}^2E_{\tau}(\lambda)\right] d\lambda,
\end{split}
\end{equation*}
which by integrating  by-parts and  observing that the boundary terms vanishes, leads to
\begin{equation*}
\begin{split} 
 \int_{a}^{\infty} g(\lambda)  \textup{Tr}\left[ V_1^2E(d\lambda)-V_{1\tau}^2E_{\tau}(d\lambda)\right] = \textup{Tr}\left[ V_1^2 h'(A)-V_{1\tau}^2 h'(A_{\tau})\right],\quad \text{which} \quad
\end{split}
\end{equation*}
by lemma \ref{lmma: decomp} $(ii)$ is equal to
\begin{equation}\label{eq: cor1}
\int_a^{\infty}\int_a^{\infty}\frac{h(\lambda)-h(\mu)}{\lambda -\mu} ~\textup{Tr}\left[V_1E(d\lambda)V_1E(d\mu) - V_{1\tau}E_{\tau}(d\lambda)V_{1\tau}E_{\tau}(d\mu)\right].
\end{equation}
Again by repeating the same calculations  as in the proof of the theorem \ref{thm: mainunbthm}, we conclude that
\begin{equation}\label{eq: cor2}
\begin{split} 
\int_{a}^{\infty} f(\lambda)\psi(\lambda) \eta_{2\tau}^{(n)}(\lambda)  d\lambda = \int_a^{\infty}\int_a^{\infty}\frac{h(\lambda)-h(\mu)}{\lambda -\mu} ~\textup{Tr}\left[V_2E(d\lambda)V_2^{(n)}E(d\mu) - V_{2\tau}E_{\tau}(d\lambda)V_{2\tau}^{(n)}E_{\tau}(d\mu)\right].
\end{split}
\end{equation}
Combining \eqref{eq: cor1} and \eqref{eq: cor2} we have,
\begin{equation}\label{eq: cor3}
\begin{split} 
\int_{-\infty}^{\infty} f(\lambda)\psi(\lambda) \eta^{(n)}(\lambda)  d\lambda \\
& \hspace{-4.2cm} = \int_0^1 ds\int_0^s d\tau\int_a^{\infty}\int_a^{\infty}\frac{h(\lambda)-h(\mu)}{\lambda -\mu} ~\textup{Tr} [(V_1 \oplus V_2) E(d\lambda)\left(V_1 \oplus V_2^{(n)}\right)E(d\mu) \\
& \hspace{4cm}  - (V_{1\tau} \oplus V_{2\tau})E_{\tau}(d\lambda)\left(V_{1\tau} \oplus V_{2\tau}^{(n)}\right) E_{\tau}(d\mu)].
\end{split}
\end{equation}
But by definition $V_{2}^{(n)}, V_{2\tau}^{(n)}$ converges to $V_2 , V_{2\tau}$ respectively in $\|.\|_2$ and we have already proved that $\eta^{(n)}$ converges to $\eta$ in $ L^1(\mathbb{R},\psi(\lambda)d\lambda)$
.Hence by taking limit on both sides of \eqref{eq: cor3} we get that 
\begin{equation}\label{eq: corendeq}
\begin{split} 
\int_{-\infty}^{\infty} f(\lambda)\psi(\lambda) \eta(\lambda)  d\lambda \\
& \hspace{-4cm} = \int_0^1 ds\int_0^s d\tau\int_a^{\infty}\int_a^{\infty}\frac{h(\lambda)-h(\mu)}{\lambda -\mu} ~\textup{Tr}\left[VE(d\lambda)VE(d\mu) - VE_{\tau}(d\lambda)V E_{\tau}(d\mu)\right],
\end{split}
\end{equation}
where we have used the fact that 
\vspace{0.1in}

\hspace{-0.8cm} $Var\left(\mathcal{G}_2^{(n)} -\mathcal{G}_2\right) \leq ~\|V\|_2 \left(\|V_{2\tau}^{(n)} - V_{2\tau}\|_2 + ~\|V_2 - V_2^{(n)}\| \right) \longrightarrow 0,$
 and that $ \|h\|_{\textup{Lip}}\leq \|g\|_{\infty} \leq \|f\|_{\infty} \|\psi\|_{L^1}.$
\vspace{0.2in}

~~~~~~~~~~~~~~~~~~~~~~~~~~~~~~~~~~~~~~~~~~~~~~~~~~~~~~~~~~~~~~~~~~~~~~~~~~~~~~~~~~~~~~~~~~~~~~~~~~~~~~~~~~~~~~~~~~~~~~~~~~~~~~~~~~~~~~~~~~\end{proof}
\hspace{-0.8cm} {\bf{\emph{ Proof of uniqueness and real-valued property of $\eta$:}}}
 For uniqueness in theorem \ref{thm: mainunbthm}, let us assume that there exists $\eta_1, \eta_2 \in L^1(\mathbb{R}, \psi(\lambda)d\lambda)$ such that 
$$\textup{Tr}\left[\phi(A+V) - \phi(A) -\left[D^{(1)}\phi(A)\right](V) - \frac{1}{2}\left[D^{(2)}\phi(A)\right](V,V)\right] = \int_{-\infty}^{\infty} \phi'''(\lambda)\eta_j(\lambda)d\lambda$$
for $j=1,2$. Then using corollary \ref{cor: 1} we conclude that
$$ \int_{\mathbb{R}} f(\lambda) \psi(\lambda) \eta_j(\lambda)d\lambda = \int_0^1ds\int_0^sd\tau \int_a^{\infty}\int_a^{\infty}\frac{h(\lambda)-h(\mu)}{\lambda -\mu} ~\textup{Tr}\left[VE(d\lambda)VE(d\mu) - VE_{\tau}(d\lambda)VE_{\tau}(d\mu)\right],$$
for $j=1,2$ and for all $f\in L^{\infty}(\mathbb{R})$. Hence 
\begin{equation}\label{eq: fstuniqeq}
\int_{\mathbb{R}} f(\lambda) ~\psi(\lambda) ~\eta(\lambda)~d\lambda = 0 ~~\forall ~~f~~\in L^{\infty}(\mathbb{R}),
\end{equation}
where  $\eta(\lambda) \equiv \eta_1(\lambda) -\eta_2(\lambda) \in L^1(\mathbb{R},\psi(\lambda)d\lambda).$ Since \eqref{eq: fstuniqeq} is true for all $f\in L^{\infty}(\mathbb{R}),$
in particular it is true for all real valued $f\in L^{\infty}(\mathbb{R})$ i.e.
\begin{equation}\label{eq: secuniqeq}
\int_{\mathbb{R}} f(\lambda) ~\psi(\lambda) ~\eta(\lambda)~d\lambda = 0 ~~\forall ~~\quad \text{real valued} \quad f~~\in L^{\infty}(\mathbb{R}).
\end{equation}
Let $\eta(\lambda) = \eta_{Rel}(\lambda) + \textup{i} ~\eta_{Img}(\lambda)$, where $\eta_{Rel}(\lambda) $ and $\eta_{Img}(\lambda) $ are real valued $L^1(\mathbb{R} ,\psi(\lambda)d\lambda)$-function.
Hence from \eqref{eq: secuniqeq} , we conclude that 
\begin{equation}\label{eq: thrduniqeq}
\int_{\mathbb{R}} f(\lambda) ~\psi(\lambda) ~\eta_{Rel}(\lambda)~d\lambda = 0 =  \int_{\mathbb{R}} f(\lambda) ~\psi(\lambda) ~\eta_{Img}(\lambda)~d\lambda~~\forall ~~\quad \text{real valued} \quad f~~\in L^{\infty}(\mathbb{R}).
\end{equation}
In particular if we consider $f(\lambda) = sgn~ \eta_{Rel}(\lambda)$, where $ sgn~ \eta_{Rel}(\lambda) = 0 ~\forall ~\lambda$ such that $\eta_{Rel}(\lambda) = 0$~;
$ sgn~ \eta_{Rel}(\lambda) = 1 ~\forall ~\lambda$ such that $\eta_{Rel}(\lambda) > 0$~; $ sgn~ \eta_{Rel}(\lambda) = -1 ~\forall ~\lambda$ such that $\eta_{Rel}(\lambda) < 0$~.
Then $f = sgn ~\eta_{Rel}\in L^{\infty}(\mathbb{R})$ and hence 
$$\int_{\mathbb{R}} \left|\eta_{Rel}(\lambda)\right| |\psi(\lambda)| ~d\lambda = \int_{\mathbb{R}} sgn~\eta_{Rel}(\lambda) ~ \eta_{Rel}(\lambda) ~\psi(\lambda) d\lambda = 0,$$
which implies that $\left|\eta_{Rel}(\lambda)\right| |\psi(\lambda)| = 0$ a.e. and hence $\eta_{Rel}(\lambda) = 0$ a.e.. Similarly by the same above argument we
conclude that $\eta_{Img}(\lambda) = 0$ a.e. and hence $\eta(\lambda) = 0$ a.e.. Therefore $\eta_1(\lambda) = \eta_2(\lambda)$a.e..
Again, since the right hand side of \eqref{eq: cor1eq} is real for all real valued $f\in L^{\infty}(\mathbb{R})$, by a similar argument as above, it follows 
that $\eta$ is real valued.
\vspace{0.2in}

\noindent\textbf{Acknowledgment:} The authors would like to thank Council of Scientific and Industrial Research (CSIR), Government of India 
for a research and Bhatnagar Fellowship respectively.

\end{document}